\documentclass{article}

\usepackage{amsmath, amsthm, amssymb}
\usepackage{color,xcolor}
\newtheorem{lemma}{Lemma}

\newtheorem{proposition}[lemma]{Proposition}
\newtheorem{remark}[lemma]{Remark}

\newtheorem{theorem}[lemma]{Theorem}

\newcommand{\e}{\varepsilon} 

\title{Local interpolation techniques for higher-order singular perturbations of non-convex functionals: free-discontinuity problems}
\author{Margherita Solci\footnote{email: {\tt margherita@uniss.it}}\\ 
{\small DADU, Universit\`a di Sassari}\\ {\small piazza Duomo 6, 07041 Alghero (Italy)}}
\date{}

\begin{document} 

\maketitle

\begin{abstract} We develop a general approach, using local interpolation inequalities, to non-convex integral functionals depending on the gradient with a singular perturbation by derivatives of order $k\ge 2$. When applied to functionals giving rise to free-discontinuity energies, such methods permit to change boundary values for derivatives up to order $k-1$ in problems defining density functions for the jump part, thus allowing to prove optimal-profile formulas, and to deduce compactness and lower bounds. As an application, we prove that for $k$-th order perturbations of energies depending on the gradient behaving as a constant at infinity, the jump energy density is a constant $m_k$ times the $k$-th root of the jump size. The result is first proved for truncated quadratic energy densities and in the one-dimensional case, from which the general higher-dimensional case can be obtained by slicing techniques. A wide class of non-convex energies can be studied as an envelope of these particular ones. Finally, we remark that an approximation of the Mumford--Shah functional can be obtained by letting $k$ tend to infinity. We also derive a new approximation of the Blake-Zisserman functional.
\bigskip



\end{abstract}


\section{Introduction}
In this paper we propose a general approach to the treatment of variational problems with non-convex energies singularly perturbed by a term with higher-order derivatives, of the form
$$
\int_\Omega \varphi(\nabla u)dx +\e^{2k-1}\int_\Omega|\nabla^{(k)}u|^2dx,
$$
where $\nabla^{(k)}u$ denotes the tensor of partial derivatives of order $k$ of $u$. The power $2k-1$ is chosen so as to have suitable scaling properties. 
 Such perturbations are commonly used to define particular solutions to problems that otherwise possess no solution or a multiplicity of solutions.
The use of higher-order derivatives, beside its theoretical interest, may be justified as
a part of a continuum higher-order expansion of atomistic models with long-range interactions (such as in \cite{MR2162866,MR1250554,BLL,BT,BCST}), where interactions beyond nearest neighbours, or deriving from many-body potentials, may be interpreted as higher-order gradients. Similarly, it can be seen as a higher-order counterpart of finite-difference continuum energies as in \cite{go,gomo}. Some proposals for continuum or quasicontinuum approximations for discrete systems can be found in \cite{CT1,CT2}. For the discussion of higher-order models in Physics and Mechanics we refer e.g.~to the book by Peletier and Troy \cite{PT}.
Another reason for the use of higher-order Sobolev spaces, is that it guarantees higher regularity of minimizers, and, as a by-product, more stable numerical approximations (for the corresponding phase-field analogs, see e.g.~\cite{Bach,BBL,BEZ,FLL}). Leaving aside topological singularities arising for vector-valued functions such as in the Ginzburg--Landau theory \cite{BBH,SS}, there are two principal examples of this approach: perturbations of double-well energies leading to phase transitions \cite{Modica}, and non-convex energies with a ``well at infinity'' such as those obtained by taking $\varphi$  Lennard-Jones or Perona--Malik potentials, leading to free-discontinuity problems \cite{abg,bf,BDS,GP-2}. Note that non-convex integrals with $\varphi$ bounded or with sublinear growth, cannot be directly used in variational theories, since their lower-semicontinuous envelope (in any useful topology) is a constant, but possess properties which are useful e.g.~for theories of Image Reconstruction. They can be chosen to be equal (or close) to the Dirichlet integral for small values of the gradient, while their derivative is small for large values of the gradient, so that discontinuity sets for $u$, representing the contour of an image, formally do not move for the gradient flow of such energies, while they follow a regularizing heat flow  elsewhere \cite{Ki}. A prototypical such integrand is the truncated quadratic potential $\varphi(z)=\min\{z^2,1\}$.

\smallskip
In this paper we develop general methods for higher-order singular perturbations, with applications focussed on free-discontinuity energies. To give a heuristic explanation of our problem, we may think of the domain of $\varphi$ as being decomposed in a convexity zone and a concavity zone. Gradient of minimizers will tend to diverge in the concavity zone as the regularizing effect of the higher-order term weakens for $\e\to 0$. In order to characterize the behaviour of such minimizers the key point is to understand the conditions they satisfy at the boundary of the concavity zone. If $k>2$ it is necessary to characterize boundary conditions also for derivatives of order $\ell\in\{2,\ldots,k-1\}$. The key issue faced in this paper is the elaboration of interpolation techniques that allow to show that, upon slightly enlarging the concavity zone, for $\varphi$ of subquadratic growth we may consider homogeneous boundary conditions for all derivatives up to order $k-1$. As a consequence this will give that the concavity zones of minimizers converge to a discontinuity hypersurface of a limit function $u$, and provide a characterization of an interfacial energy density on that discontinuity hypersurface determined by the size of its jump discontinuity.
As an application and motivation for our interpolation techniques, we will give an answer to a long-standing question arisen from \cite{abg} for singular perturbations of order $k$ larger or equal than $3$ (in that paper only the case $k=2$ was solved) of energies involving the truncated quadratic potential (see also \cite{GP-2} Section 7), for which heuristic scaling arguments suggest a multiple of the $k$-th root of the jump size as a surface energy density.

\medskip
In order to use a variational approach, a scaling argument (formalized in \cite{BLO} using a renor\-mali\-zation-group approach) suggests to consider functionals of the form
$$
\frac1\e\int_\Omega \varphi(\sqrt\e\nabla u)\,dx +\e^{2k-1}\int_\Omega |\nabla^{(k)}u|^2\,dx.
$$
Minimizers for minimum problems involving the original energy can be derived from minimizers of this one by a scaling argument,
and the scaling ensures that such energies have a non-trivial limit in the sense of $\Gamma$-convergence. We will concentrate on the truncated quadratic potential, for which we consider
\begin{equation}\label{intro:0}
F_\e(u)=
\int_\Omega \min\Big\{|\nabla u|^2,\frac1\e\Big\}\,dx +\e^{2k-1}\int_\Omega \|\nabla^{(k)}u\|^2\,dx,
\end{equation}
with $\|\nabla^{(k)} u\|$, denoting the operator norm of the tensor of the partial derivatives of order $k$, chosen for technical convenience. 
Note that in order to overcome the non-convexity of the first term, that integral can be discretized as in the works of Chambolle \cite{cha} or transformed in a truncated version of a  Gagliardo seminorm as in a late conjecture by De Giorgi proved by Gobbino \cite{go}, or mollified  by taking some local averaged gradient \cite{BDM}.
Such energies are shown to $\Gamma$-converge to some version of the Mumford--Shah functional of Computer Vision \cite{MS} as $\e$ tends to $0$, thus also ensuring the convergence of the related minimum problems. We will prove an analogous result for the energies $F_\e$.
A first step is to show that the analysis can be reduced to the one-dimensional case; that is, we may only consider one-dimensional functionals
\begin{equation}\label{intro-4}
F_\e(u)=\int_{(0,1)} \min\Big\{|u'|^2,\frac1\e\Big\}\,dx +\e^{2k-1}\int_{(0,1)} |u^{(k)}|^2\,dx,
\end{equation}
for simplicity parameterized on $(0,1)$. Then, in our approach, given a sequence $u_\e$ with equibounded energy, the `convexity zones' for $u_\e$ are intervals where $|u_\e'|^2\le\frac1\e$. To analyze such intervals, the main tool that we use is an interpolation inequality (applied to $u=u_\e$) as follows (see Lemma \ref{lemmainterp}).

\smallskip
\noindent{\bf Lemma.} {\em 
There exists a constant $R_k>0$ depending only on $k$ such that for any $\ell\in\{2,\dots, k-1\}$  and $\e\in(0,1)$, and  for any interval $I$ such that  $u\in H^k(I)$
and $|u'|^2\le\frac1\e$ on $I$  the following estimate holds 
\begin{equation}\label{interpolation-0}
\|u^{(\ell)}\|^2_{L^2(I)}\leq R_k\Big(\e^{-\frac{2k-1}{k-1}(\ell-1)} F_\e(u,I)+{|I|^{-2(\ell-1)}} \|u^{\prime}\|^2_{L^2(I)}\Big), 
\end{equation}
where $F_\e(u,I)$ is the localized functional defined as in \eqref{intro-4} with $I$ in the place of $(0,1)$.} 

\smallskip
This interpolation allows, instead of considering the problems on `concavity intervals' (which in this case are simply intervals where $|u_\e'|$ is above the threshold $\frac1{\sqrt\e}$) to prove that slightly larger intervals exist at the endpoints of which all higher derivatives of $u_\e$ are small; that is, for any fixed $\eta>0$, we can suppose that $|u^{(\ell)}_\e|\le \frac\eta{\e^\ell}$ at the endpoints for all $\ell\in\{1,\ldots,k-1\}$.  The issue is to control the relative length of the added intervals so that their total lengths can be neglected in the limit. The argument is particularly delicate since the number of transition intervals and their lengths are not controlled. For this purpose a very careful analysis, which is the main and most technical part of the work, has to be carried on by iteratively using interpolation inequalities. We note the local nature of these techniques, which allows to exploit additional properties of $u_\e$ on these intervals and does not require global assumptions. Moreover, even though the lemma above is stated for the truncated quadratic potential, it has been adapted for double-well potentials. In that case, on one hand sets where derivatives are below the threshold have to be substituted with sets where the functions themselves are close to the wells, while a simplification is that a single optimal profile has to be characterized.

Using this interpolation argument for the characterization of interfacial energy densities, 
we may describe the asymptotic behaviour of a wider class of functionals than \eqref{intro:0}, defined as follows, in a notation suggested by De Giorgi \cite{go},
and  prove the following result (Theorem \ref{main}, Remark \ref{main-d} and Proposition \ref{main-gen}).

\smallskip 
\noindent{\bf Theorem.}  {\em Let $f$ be a lower-semicontinuos increasing function on $[0,+\infty)$ such that $f(0)=0$, and there exist $a,b>0$ such that $f'(0)=a$ and $f(z)\to b$ as $z\to+\infty$. Let $\Omega$ be a bounded open subset of $\mathbb R^d$ and let  
\begin{equation}\label{intro-17}
F_\e(u)=\int_{\Omega}\Big(\frac{1}{\e}f(\e |\nabla u|^2)+  \e^{2k-1}\|\nabla^{(k)} u\|^2\Big)\, dt
\end{equation}
be defined on $H^k(\Omega)$. Then the $\Gamma$-limit in the $L^1(\Omega)$-topology of $F_\e$ is the free-discontinuity functional  $F$ defined on $GSBV(\Omega)$ as
\begin{equation}\label{intro-6}
F(u)=a\int_{\Omega}|\nabla u|^2\,dt + m_k(b)\int_{S(u)\cap\Omega}|u^+-u^-|^{\frac{1}{k}} d{\mathcal H}^{d-1},
\end{equation}
with $S(u)$ denoting the jump set of $u$, and $u^\pm$ the traces of $u$ on both sides of $S(u)$.
The strictly positive constant $m_k(b)$ is given by
\begin{eqnarray}\label{intro-7}\nonumber
&&\hskip-1cm m_k(b)=\min_{T>0}\min\bigg\{b\, T +\int_{-\frac{T}{2}}^{\frac{T}{2}} (v^{(k)})^2\,dt: v\in H^k\big(-\tfrac{T}{2},\tfrac{T}{2}\big), \\
&& \hskip2cm v\big(\pm\tfrac{T}{2}\big)= \pm \tfrac{1}{2}, v^{(\ell)}\big(\pm\tfrac{T}{2}\big)=0 \hbox{ if } \ell\in\{1,\ldots, k-1\}\Big\}.
\end{eqnarray}}

We highlight that in the last formula the homogeneous boundary conditions on $v^{(\ell)}$ for $\ell\in\{2,\ldots, k-1\}$ are not trivially obtained from the functionals, and are the outcome of our interpolation techniques.

Note that functional  \eqref{intro:0} is a particular case of functional \eqref{intro-17} obtained by taking $f(z)=\min\{|z|,1\}$, and conversely functional \eqref{intro-17} can be seen as an envelope of functionals of the type \eqref{intro:0} taking $f(z)=a'\min\{|z|,b'/a'\}$ for a family of $a',b'$. This implies that it is sufficient to treat case \eqref{intro:0} both for proving coerciveness and to obtain lower bounds. Moreover, the use of the operator norm in the perturbation allows a blow-up and slicing argument \cite{fmu,bln} showing that it is sufficient to treat the one-dimensional case.

\smallskip
As a by-product of this theorem, we can obtain a new approximation of the Mumford--Shah functional \cite{MS}
by functionals
\begin{equation*}
F_k(u)=\int_{\Omega}\big(\tfrac{1}{\e_k}f(\e_k |\nabla u|^2)+ c_k \e_k^{2k-1}\|\nabla^{(k)}u\|^2\big)\, dt 
\end{equation*}
for suitable $\e_k\to 0$. The parameter $c_k$ is determined so as to normalize $m_k(b)$ to a given constant for all $k$. This result follows from the theorem above using a diagonal argument, after noting that for $z\neq0$ we have that $|z|^{{1}/{k}}$ tends to $1$ as $k\to+\infty$.
In turn, from this result, also an approximation of the Blake--Zisserman functional of Visual Reconstruction \cite{BZ} can be obtained.

\smallskip
The central issue in the asymptotic analysis of functionals $F_\e$ is the possibility of estimating the energy of optimal transitions at jump points of a limit function $u$. The core {\em ansatz} for such an estimate is double: first, as already remarked, that we can reduce to one-dimensional profiles,
and then that we may restrict to functions whose derivatives from the second to order $k-1$ are $0$ at the boundary of {\em transition intervals} $(x_\e,x_\e+\e T_\e)$, defined as maximal intervals where $u'$ is in the non-convexity zone; that is, where $|u'|^2> 1/\e$. After a scaling argument of the form $v(t)=u(x_\e+\e t)$, and optimization in $T_\e$ this gives that approximately the optimal-transition energy for a jump of amplitude $z$ is given by  
\begin{eqnarray}\label{intro:100}\nonumber
&&\varphi_\e(z):=\min_{T,v}\bigg\{ T+\int_0^T|v^{(k)}|^2\,dt: v(T)-v(0)=z, v'(0)=v'(T)\in\{\pm\sqrt\e\},\\
&& \hskip3cm v^{(\ell)}(0)=v^{(\ell)}(T)=0 \hbox{ for } \ell\in\{2,\ldots,k-1\}\Big\},
\end{eqnarray}
from which a sharp lower bound is obtained by letting $\e$ tend to $0$.
Note that this ansatz is trivial if $k=2$ since in this case there is no condition on derivatives higher that $1$ and the condition on the first derivative of $v$ is a consequence of $u'\in\{\pm1/{\sqrt{\e}}\}$ at the endpoints of a transition interval. That case has been studied in \cite{abg} showing that the jump energy density is $m_2\sqrt{|z|}$, with $m_2$ explicitly given.

The main point in validating this ansatz for $k>2$ is the possibility of modifying the scaling argument so as to obtain homogeneous conditions $v^{(\ell)}=0$ at the boundary of $(0,T)$ for the higher-order derivatives of test functions. 
As for many problems in the Calculus of Variations, the
possibility of changing boundary values is of central importance since it allows to describe averaged or concentrating quantities by means of homogenization formulas or optimal-profile problems. 
In our case, the argument could be carried on if we proved that at the endpoint of a transition interval all derivatives up to order $k-1$ were small (after scaling). This seems very difficult, if true, to prove directly. We then relax our ansatz transforming it in an asymptotic problem on a sequence $u_\e$ with equibounded $F_\e(u_\e)$, and using the lemma above.

The application of our interpolation techniques allows to use as a lower bound the functions $\varphi_\e$ as in \eqref{intro:100}, and hence their limit $\varphi$ as $\e\to 0$. A change of variables show that $\varphi(z)=m_k\root k\of {|z|}$, with $m_k=m_k(1)$ as in \eqref{intro-7}. We note that this argument can only be applied for $|z|\ge \theta$ for each $\theta>0$ fixed, but as such it is not sufficient to prove $\Gamma$-convergence, since the lower semicontinuity of functionals on $SBV$ is affected by the behaviour of jump energy densities at $z=0$. Hence, for small jumps a different energy density has to be constructed; in particular, an interesting observation is that minimal-transition energies of the form \eqref{intro:100} but with conditions on the derivatives only for $\ell$ with $2<2\ell\le k$ still provide a coercive lower bound, so that, as long as growth estimates are needed, fewer conditions are required. With the use of these energy densities, for each given sequence $u_\e$ we construct a sequence $w_\e\in SBV(0,1)$ with the same limit, and such that $F_\e(u_\e)\ge G_\e(w_\e)$ for a family of functionals $G_\e$, which can be explicitly described on $BV(0,1)$ \cite{bbb}, whose $\Gamma$-limit is indeed $F$. It is worth noting that in this way the coerciveness of the energies $F_\e$ is obtained indirectly through a coerciveness argument in $BV$ and proving separately that the limit is in $SBV$.  Once the lower bound with $m_k$ in \eqref{intro-7} is proved, an upper bound can be constructed using optimal-transition functions $v$ since the scaled functions $t\mapsto v({t}/{\e})$, converging to a jump, can be extended to functions with equibounded energy. We note the surprising level of complexity with respect to the case $k=2$ of the case $k\ge 3$, whose proof seems not to be simplifiable.

\smallskip
The paper is organized as follows. In Section \ref{sec:stat} we state the `basic' $\Gamma$-con\-ver\-gence result Theorem \ref{main} for perturbations of functionals  \eqref{intro-4} with truncated-quadratic potentials in dimension one, defining the optimal-profile problem \eqref{intro-7} giving the constant $m_k$, and briefly recall how the corresponding $d$-dimensional case is recovered by slicing and by using an approximation argument of \cite{ag}. The derivation as a corollary of a more general result for the functionals in \eqref{intro-17} is postponed to Section \ref{subsec:general}. In Section \ref{optimalprofile} we also introduce some auxiliary optimal-profile problems that approximate or estimate the one giving the constant $m_k$. Section \ref{inter} contains the general local interpolation results estimating the length intervals `below the threshold' in terms of the energy and the behaviour of higher-order derivatives in them. The section is subdivided into the analysis of `large' and `small' intervals. These results are used in Section \ref{sec:proof} to prove the convergence theorem, by a very careful use of several approximate optimal-profile problems, that allow a comparison with a sequence of free-discontinuity energies, which we can study thanks to a result recalled in Section \ref{sec:scBV}. Finally, in Section \ref{sec:ext} we also provide two new approximations, of the Mumford--Shah and Blake--Zisserman functionals.

\smallskip
We conclude the Introduction with some remarks and perspectives.
Our local interpolation methods, which  only use lower bounds to determine compactness properties and transition formulas, are quite flexible and can be used to treat general functions with sublinear growth such as logarithmically growing Perona--Malik energies, as already done for the case $k=2$ in \cite{GP-2} for a different scaling (see also \cite{bcg,bf,GP-2,mo}). Furthermore, the perturbation term can be generalized taking, in addition to $k$-th order derivatives, also a linear combination of the norms of derivatives of order $\ell$ between $2$ and $k-1$, scaled by $\e^{2\ell-1}$. It would be interesting to explore the case when the coefficients of such lower-order perturbations may be slightly negative, in the spirit of analogous analyses for phase-transition energies \cite{CDFL,CSZ}. In both cases, the surface energy density will not be positively homogeneous, and its properties could be an object of analysis, to determine which surface energy densities can be obtained by this limit procedure.  
We also note that similar questions have arisen for models of phase transitions \cite{fm,CDFL,CSZ}, for which, starting from the present work, our methods have already been adapted and simplified \cite{BruDoS}. Some of the other directions in which our methods could be developed are the approximation of classes of free-discontinuity problems by taking more general non-monotone energy densities (see \cite[Section 3.3.3]{bln} and recent work on the approximation of cohesive fracture models; see e.g.~\cite{CFI1,CFI2}), and the study of the effect of more general perturbations using non-local Gagliardo seminorms, as done for phase-transition problems in \cite{SV} (see also \cite{ABSS} for a perturbation approach leading to Ginzburg--Landau vortices).
In particular, if we use perturbations with the square of a scaled fractional seminorm of $\nabla^{(k)}u$ in the Sobolev space $H^s$ with $s\in(0,1)$, a scaling argument suggests that the limit surface energy is a multiple of the $\frac1{k+s}$ power of the jump size.

\section{Statement and preliminaries}\label{sec:stat}

As noted in the Introduction, the relevant analysis is carried on in dimension one, and for a special sequence of functionals, since the study all other functionals  can be reduced to dimension one by slicing and to envelopes of limits of sequences as below.

For $k\ge2$ we define the functional  
\begin{equation}\label{defFe}
F_\e(u)=\int_{(0,1)}\big(f_\e(u^\prime)+ \e^{2k-1}(u^{(k)})^2\big)dt 
\end{equation}
with $u\in H^k(0,1)$, where 
$$f_\e(z)=\min\Big\{z^2,\frac{1}{\e}\Big\}.$$
The choice of $(0,1)$ as a domain for the integrals is simply due to notational convenience.
The results are valid on a generic bounded interval.

The domain of the $\Gamma$-limit of $F_\e$ will be a set of special functions of bounded variations,
more precisely, a subspace of the space $SBV_2(0,1)$ of functions $u\in BV(0,1)$ whose distributional derivative 
is a measure that can be written as 
$$
Du= u'{\mathcal L}^1 + \sum_{t\in S(u)} (u(t+)-u(t-))\delta_t,
$$
and $u'\in L^2(0,1)$. In this notation, $u'$ is the approximate gradient of $u$, $S(u)$ is the set of jump points of $u$ and $u(t+)$ and $u(t-)$ are the right-hand and left-hand limits at $t$, respectively (see e.g.~\cite{bln}).
\medskip

The limit behaviour of $F_\e$ is described by the following result.

\begin{theorem}[$\Gamma$-convergence and coerciveness]\label{main} 
The functionals \eqref{defFe} $\Gamma$-converge with respect to the convergence in measure and with respect to the $L^1(0,1)$-convergence to the functional
\begin{equation}\label{limitk}
F(u)=\int_{(0,1)}(u^\prime)^2\, dt + m_k\sum_{t\in S(u)}|u(t+)-u(t-)|^{\frac1k}
\end{equation}
with domain $SBV_2(0,1)$. The strictly positive constant $m_k$ is defined by the {\em op\-ti\-mal-profile problem}
\begin{eqnarray}\label{defCk} 
&&\hspace{-12mm}m_k=\inf_{T>0}\min\Bigg\{T+\int_{-\tfrac{T}{2}}^{\tfrac{T}{2}} (v^{(k)})^2\, dt: v\in H^k\big(-\tfrac{T}{2}, \tfrac{T}{2}\big), \ v\big(\pm\tfrac{T}{2}\big)=\pm\tfrac{1}{2}\nonumber\\ 
&&\hspace{2cm} \hbox{\rm and } \ v^{(\ell)}\big(\pm\tfrac{T}{2}\big)=0 \ \hbox{\rm for all } \ell\in\{1,\ldots,k-1\}\Big\}. 
\end{eqnarray}
Furthermore, the functionals are equicoercive, in the sense that if $\sup_{\e} F_\e(u_\e)<+\infty$ then, up to addition of constants and up to subsequences, there exists $u\in SBV_2(0,1)$ such that $u_\e$ tends to $u$ in measure as $\e\to0$.
\end{theorem}

\begin{remark}[The higher-dimensional case]\label{main-d}\rm 
Theorem \ref{main} extends to dimension $d\ge 2$, upon defining the functionals as
\begin{equation}\label{defFe-d}
F_\e(u)=\int_{\Omega}\Big(f_\e(|\nabla u|)+ \e^{2k-1}\|\nabla^{(k)} u\|^2\Big)\, dx 
\end{equation}
on $H^k(\Omega)$, with $\Omega\subset\mathbb R^d$ a bounded open set with Lipschitz boundary.
In order to reduce to the one-dimensional case, the norm of the tensor of the $k$-th order derivatives is the operator norm; i.e., 
$$
\|\nabla^{(k)} u\|=\sup \bigg\{\bigg|\sum_{i_1,\ldots, i_k=1}^d \frac{\partial ^k u}{\partial x_{i_1}\cdots\partial x_{i_k}}\xi_{i_1}\cdots\xi_{i_k}\bigg|: |\xi|=1\bigg\}.
$$
For these functionals we can proceed exactly as in the paper by Alicandro and Gelli \cite{ag} to which we refer for the notation and more details. The domain of the $\Gamma$-limit with respect to the convergence in measure is a subset of the space of generalized special functions of bounded variation $GSBV_2(\Omega)$ of functions whose truncations are in $SBV_2(\Omega)$; that is they are $SBV(\Omega)$ with $\nabla u\in L^2(\Omega;\mathbb R^d)$ (see e.g.~\cite{bln}), and can be written as  
$$
F(u)=\int_{(0,1)}|\nabla u|^2\, dx + m_k\int_{S(u)}|u^+(x)-u^-(x)|^{\frac1k}\, d\mathcal H^{d-1}(x),
$$
where now $u^\pm$ denote the traces of $u$ on both sides of the jump set $S(u)$, and $ \mathcal H^{d-1}$ is the ($d-1$)-dimensional Hausdorff measure. 
In order to use the results of \cite{ag}, we have to compute the $\Gamma$-limit in the $L^2$-topology. Note that this is not a restriction in many problems where a {\em fidelity term} is added to the functional, of the form $\lambda\int_\Omega |u-g_\e|^2dx$, where $g_\e$ is a converging sequence in $L^2(\Omega)$. This is a usual addition for problems in Image Reconstruction or for minimizing-movements schemes along the family $F_\e$ (see \cite{AGS,BS-book}); moreover, it allows to use a Fr\'echet-Kolmogorov compactness argument to prove equi-coerciveness in $L^1(\Omega)$ (see e.g.~\cite[Section 6.3]{ABS}). Using this topology, the proof can be obtained by following very closely that in \cite{ag}. The lower bound is obtained by a slicing argument, comparing with one-dimensional energies for functions of the form $t\mapsto u_\e(y+t\xi)$. To this end, it is necessary to use the operator norm. The upper bound is obtained by first dealing with $u=\chi_E$, with $E$ a set with a smooth boundary. In this case, the construction of a recovery sequence is the same as for the case $k=2$, defining $u_\e(x)= v(\frac{1}{\e} d(x))$ on an $\e\frac{T}2$-neighbourhood of $\partial E$, where $(T,v)$ is an optimal pair for the definition of $m_k$ and $d$ is the signed distance from $\partial E$.  After this case is dealt with, we may proceed as in \cite{ag}, first extending the construction of a recovery sequence for $u$ a finite combination of characteristic functions, and finally using a density result that allows to approximate the target $u$  with a sequence $u_j$, with $u_j$ that is a finite combination of characteristic functions on a small set close to the jump set of $u$ and smooth otherwise. It is interesting to note that actually we can improve the result of \cite{ag}, which was proved for functions with $\mathcal H^{d-1}(S(u))<+\infty$, since we have at our disposal a recent finer approximation result by Conti, Focardi and Iurlano \cite{cfi}, which removes this restriction.

The coerciveness properties of the functionals \eqref{defFe-d} on bounded sets of $L^2(\Omega)$ hold by comparison also if we replace the operator norm with any other norm. In general, in that case the slicing argument does not hold and it is likely that a different formula can be obtained using the blow-up technique of Fonseca--M\"uller \cite{fmu} (or some of its variants for $\e$-depending functionals as in \cite{BMS}).
\end{remark}

\subsection{Semicontinuity and relaxation in $BV$}\label{sec:scBV}
A step in the proof of Theorem \ref{main} will be the comparison of $F_\e(u_\e)$ with some $G_\e(v_\e)$, where $v_\e$ are suitable $SBV_2$ functions constructed from $u_\e$ with the same limit. We will need the following relaxation result.

\begin{lemma}[Theorem 3.1, \cite{bbb}]\label{bbblemma}
Let $g\colon[0,+\infty)\to[0,+\infty)$ be a strictly increasing concave function with $g(0)=0$ and $\lim\limits_{z\to+\infty}g(z)=+\infty$, and consider the functional  $G$  defined by 
$$G(v)=\sum_{t\in S(v)}g(|v(t+)-v(t-)|)+\int_{0}^{1} (v^\prime)^2\, dt$$ 
 if $v\in SBV_2(0,1)$ and $G(v)=+\infty$ otherwise in $BV(0,1)$. 
Then, the lower semicontinuous envelope of $G$ with respect to the $L^1$-convergence is given by 
\begin{equation*} 
\overline G(v)=\sum_{t\in S(v)}g(|v(t+)-v(t-)|)+\int_{0}^{1} \gamma_g^{\ast\ast}(v^\prime)\, dt
+g^\prime(0)|D_c v|(0,1),
\end{equation*}
where 
$\gamma_g(z)=\min\{z^2, g^\prime(0)|z|\}^{\ast\ast}$. 
\end{lemma}

This lemma will be applied to $G=G_\e$ such that the corresponding $g_\e$ satisfy $g_\e^\prime(0)\to +\infty$ as $\e\to 0$.

Finally, we recall that an increasing sequence of lower-semicontinuous functionals $\Gamma$-converges to its pointwise limit. This, together with the property that the $\Gamma$-limit of a sequence of functionals is the same as that of their lower-semicontinous envelopes, will be applied to characterize the limit of $G_\e$ as $\e\to 0$. 

\subsection{Jump energy densities}\label{optimalprofile} 
In the proof of the equicoerciveness of the sequence $\{F_\e\}$, we will use the existence of strictly positive approximate energy densities. The possibility of imposing boundary conditions on derivatives up to order $k-1$ will be ensured by the application of the interpolation results.

\begin{remark}[Approximate surface energy densities]\label{MnM}\rm Let $k$ and $n$ be integers with $k\geq 2$  
and  
 $2n\ge k$. For all $N\geq 1$ we set 
\begin{eqnarray}\label{defCN}
&&\hspace{-8mm}m^n_k(N)=\inf_{T>0}\min\bigg\{T+\int_{-\frac{T}{2}}^{\frac{T}{2}} (v^{(k)})^2\, dt: v\in H^k\big(-\tfrac{T}{2}, \tfrac{T}{2}\big), \ v\big(\pm\tfrac{T}{2}\big)=\pm\tfrac{1}{2}   \nonumber\\
&&\hspace{2.5cm} \ \hbox{\rm and }\ \big|v^{(\ell)}\big(\pm\tfrac{T}{2}\big)\big|\leq\tfrac{1}{N} \ \hbox{\rm for all } \ell\in\{1,\ldots,n\}\Big\}. 
\end{eqnarray} 
Then we have $m^n_k(N)>0$ for all $N$. Indeed, otherwise there exist $T^N_j\to 0$ as $j\to+\infty$, and test functions $v^N_j$ such that 
$$\lim_{j\to+\infty} \int_{-\frac{T^N_j}{2}}^{\frac{T^N_j}{2}} ((v^N_j)^{(k)})^2\, dt=0.$$
We scale the functions to $(-\frac{1}{2},\frac{1}{2})$ by setting $w^N_j(t)= v^N_j(T^N_j t)$, and note that
$$\lim_{j\to+\infty} \int_{-\frac{1}{2}}^{\frac{1}{2}} ((w^N_j)^{(k)})^2\, dt=0.$$
Using iteratively the Poincar\'e-Wirtinger inequality we deduce that there exist polynomials $P_j$ of degree at most $k-1$ such that 
$$
\lim_{j\to+\infty}\|w^N_j-P_j\|_{H^k(-\frac{1}{2},\frac{1}{2})}=0.
$$
Note that $(w^N_j)^{(\ell)}(t)= (T^N_j)^\ell (v^N_j)^{(\ell)}(T^N_j t)$, so that, by the convergence in $C^{k-1}([-\frac{1}{2},\frac{1}{2}])$ of $w^N_j-P_j$ to $0$, in particular we have
\begin{equation}\label{system-1}
\lim_{j\to+\infty}P^{(\ell)}_j\big(\pm\tfrac12\big)=\lim_{j\to+\infty}(w^N_j)^{(\ell)}\big(\pm\tfrac12\big)= 0 \quad \hbox{ for all } \ell\in\{1,\ldots, n\}\,.
\end{equation}
Note that \eqref{system-1} provides 
a sequence of linear systems 
in the $k-1$ coefficients of $P_j^\prime$,
with the right-hand side tending to $0$. 
Since $2n\geq k$, the number of independent equations in these systems is exactly $k-1$. 
This shows that $P_j^\prime$ 
tend to $0$ as $j\to+\infty$, 
giving a contradiction since 
\begin{equation*}
\lim_{j\to+\infty}P_j\big(\pm\tfrac12\big)=\lim_{j\to+\infty}w^N_j\big(\pm\tfrac12\big)= \pm\tfrac12.
\end{equation*}
\end{remark}
In the special case $n=k-1$, we set $m_k(N)=m^{k-1}_k(N)$; that is, 
\begin{eqnarray}\label{def-MkN}
&&\hspace{-8mm}m_k(N)=\inf_{T>0}\min\bigg\{T+\int_{-\frac{T}{2}}^{\frac{T}{2}} (v^{(k)})^2\, dt: v\in H^k\big(-\tfrac{T}{2}, \tfrac{T}{2}\big), \ v\big(\pm\tfrac{T}{2}\big)=\pm\tfrac{1}{2}   \nonumber\\
&&\hspace{2.4cm} \ \hbox{\rm and }\ \big|v^{(\ell)}\big(\pm\tfrac{T}{2}\big)\big|\leq\tfrac{1}{N} \ \hbox{\rm for all } \ell\in\{1,\ldots,k-1\}\Big\}. 
\end{eqnarray} 
\begin{lemma}[Convergence of approximate surface energy densities]\label{constant} 
The sequence $m_k^n(N)$ increasingly converges to 
\begin{eqnarray}\label{defCkn} 
&&\hspace{-12mm}m_k^n:=\inf_{T>0}\min\bigg\{T+\int_{-\frac{T}{2}}^{\frac{T}{2}} (v^{(k)})^2\, dt: v\in H^k\big(-\tfrac{T}{2}, \tfrac{T}{2}\big), \ v\big(\pm\tfrac{T}{2}\big)=\pm\tfrac{1}{2}\nonumber\\ 
&&\hspace{2cm} \hbox{\rm and } \ v^{(\ell)}\big(\pm\tfrac{T}{2}\big)=0 \ \hbox{\rm for all } \ell\in\{1,\ldots,n\}\Big\}. 
\end{eqnarray}
as $N\to+\infty$; in particular, for $n=k-1$, we have that $m_k(N)\to m_k$ defined by \eqref{defCk}. As a consequence, $m^n_k$ and 
$m_k$ are strictly positive. 
\end{lemma}
\begin{proof} By definition, $m^n_k(N)$ is increasing in $N$, and $m^n_k(N)\le m^n_k$, so that 
$$
\limsup_{N\to+\infty} m^n_k(N)\le m^n_k.
$$

Conversely, let $(T_N,v_N)$ be a minimizing pair for $m^n_k(N)$. 
By Remark \ref{MnM}, up to subsequences we can suppose that $T_N\to T_0\in (0,+\infty)$. 
Up to an asymptotically negligible scaling, we can suppose that $T_N=T_0$, 
and $v_N$ tend to some $v_0$ weakly in $H^k(-\frac{T_0}{2},\frac{T_0}{2})$. By the convergence to $0$ of the derivatives up to order $n$, $(T_0,v_0)$ is a test pair for $m^n_k$. 
By the lower semicontinuity of the $L^2$-norm we then have 
$$
m^n_k\le T_0+\int_{-\frac{T_0}{2}}^{\frac{T_0}{2}} (v_0^{(k)})^2\, dt\le \liminf_{N\to+\infty} m^n_k(N),
$$
which proves that $m^n_k$ is the limit of $m^n_k(N)$, as desired.
\end{proof}

\section{Interpolation methods for higher-order derivatives `below the threshold'}\label{inter}
In this section we develop a way to estimate the sets where a function $u_\e$ is `below the threshold'; that is, $|u'_\e|^2\le 1/\e$, in terms of the value of their higher-order derivatives up to $k-1$, with the objective to show that sets where all derivatives are `small' are sufficiently uniformly dense. 

We assume $k\ge 3$ since the case $k=2$ has been dealt with in \cite{abg}.
We will use the following notation for the localized version of the functionals
$$
F_\e(u;I)=\int_{I}\Big(f_\e(u^\prime)+ \e^{2k-1}(u^{(k)})^2\Big)\, dt 
$$
on a subset $I$ of $(0,1)$.

We fix a family $\{u_\e\}$ such that 
\begin{equation}\label{defS}
S:=\sup_{\e>0} F_\e(u_\e)<+\infty.
\end{equation}
Since $u_\e^\prime$ is continuous, the set where $u^\prime_\e$ is below the threshold can be written as 
$$\mathcal A_\e:=\Big\{t\in (0,1): |u_\e^\prime(t)|<\frac{1}{\sqrt\e}\Big\}=\bigcup_{j\in \mathcal I_\e}I_j,$$
where $\mathcal I_\e$ is a countable set of indices,  
each $I_j=(a_j,b_j)$ is a maximal open interval; that is, 
 either $|u^\prime_\e|=\frac{1}{\sqrt{\e}}$ at the endpoints or the endpoints belong to $\{0,1\}$, and the intervals are pairwise disjoint. 
 
 Note that by the equiboundedness of $F_\e(u_\e)$ we have 
 \begin{equation}\label{menoeps}
 |(0,1)\setminus \mathcal A_\e|=\Big|\Big\{t\in (0,1): |u_\e^\prime(t)|\geq\frac{1}{\sqrt\e}\Big\}\Big|\leq \e S. 
 \end{equation}

The heuristic argument in order to give an estimate of the energies on each maximal interval of $(0,1)\setminus \mathcal A_\e$ strictly contained in $(0,1)$ is that such an interval can be extended to an interval $(\tau,\sigma)$ with all derivatives small up to order $k-1$ at the endpoints up to adding small intervals inside $\mathcal A_\e$. By optimizing on all functions satisfying those boundary conditions we obtain an estimate of the energy sufficient to carry out a compactness argument for $u_\e$. 
 
Note that the reasoning above works if the additional small intervals inside $\mathcal A_\e$ are small or comparable with their total length. The estimate of the length of such intervals is the central point of Section \ref{largeint}.
It will then be proved in Section~\ref{sec42} that such additional intervals are relatively small if $|u_\e(\sigma)-u_\e(\tau)|\ge\theta>0$ for some $\theta$ uniformly in $\e$. If, otherwise, $|u_\e(\sigma)-u_\e(\tau)|$ is small, then we have to take into account that we do not have a control on $\sigma-\tau$, and  the argument is more involved. It is then necessary to introduce a suitable threshold $r\in(0,1)$ and subdivide the intervals of $\mathcal A_\e\cap (\tau,\sigma)$ into intervals $I$ where either $\int_I |u_\e|^2dx> \frac{r}\e|I|$ (and hence can be considered as intervals where $u_\e$ is `above a smaller threshold'), or where sufficiently many derivatives of $u_\e$ are small at the endpoints, again allowing an estimate of the energy sufficient to deduce compactness properties for $u_\e$. The analysis of intervals 
$I$ with $\int_I |u_\e|^2dx\le \frac{r}\e|I|$ and the possibility of the subdivision of $\mathcal A_\e\cap (\tau,\sigma)$ just described is the central point of Section \ref{smallint}.

\subsection{Analysis of intervals `with large derivatives'}\label{largeint}
In this section we estimate the size of the intervals $I_j$ that do not contain points where all derivatives of $u_\e$ up to order $k-1$ are small. 
This will be a technical tool used in the proof of the coerciveness and of the lower estimate for the $\Gamma$-limit. 

To describe such intervals, for any fixed $N\geq 1$ and 
for any $\e>0$ we introduce the set $\mathcal I_\e(N)$ of the indices $j\in \mathcal I_\e$ such that 
\begin{equation}\label{defIN}
\Big|\Big\{t\in I_j: |u_\e^{(\ell)}(t)|<\frac{1}{N\e^\ell} \ \hbox{\rm for all } \ell\in\{2,\dots, k-1\}\Big\}\Big|>0.
\end{equation} 

We now estimate the length of intervals $I_j$ such that $j\in \mathcal I_\e\setminus \mathcal I_\e(N)$. 
\begin{lemma}[Estimate for intervals with large derivatives]\label{lowerlemma} 
Let $N\geq 1$ be fixed. 
There exists a positive constant $R_k(N)$ such that 
for any $\e\in(0,1)$ and $I\subseteq (0,1)$ open interval satisfying 
\begin{enumerate}
\item[{\rm (i)}] $|u_\e^\prime(t)|<\frac{1}{\sqrt\e}$ for all $t\in I$; 
\item[{\rm (ii)}] the set $\{t\in I: |u_\e^{(\ell)}(t)|<\frac{1}{N\e^\ell} \ \hbox{\rm for all } \ell\in\{2,\dots, k-1\}\}$ has measure $0$, 
\end{enumerate} 
it holds 
$$|I|\leq R_k(N) 
\e^{\frac{2k-3}{2k-4}}.$$ 
In particular, the estimate holds for all $\e\in(0,1)$ and $I_j$ with $j\in \mathcal I_\e\setminus \mathcal I_\e(N)$. 
\end{lemma}

The proof is based on the following interpolation lemma. 
\begin{lemma}[Interpolation]\label{lemmainterp} 
Let $k\geq 3$ be fixed. 
There exists a constant $R_k>0$ (depending only on $k$) such that for any bounded interval $I$, $u\in H^k(I)$, $\ell\in\{2,\dots, k-1\}$  and $\e\in(0,1)$ the following estimate holds 
\begin{equation}\label{interpolation}
\e^{\gamma_\ell}\|u^{(\ell)}\|^2_{L^2(I)}\leq R_k\Big(\|u^{\prime}\|^2_{L^2(I)}+\e^{2k-1}\|u^{(k)}\|^2_{L^2(I)}+\frac{\e^{\gamma_\ell}}{|I|^{2(\ell-1)}} \|u^{\prime}\|^2_{L^2(I)}\Big), 
\end{equation}
where $\gamma_\ell=\frac{2k-1}{k-1}(\ell-1)$. 
\end{lemma} 

\begin{proof}
We recall a classical interpolation inequality stating that, 
given $n,r\in\mathbb N$ with $n>r$, there exists a constant $C(r,n)>0$ such that 
for any $I$ open bounded interval and $v\in H^n(I)$ the following estimate holds 
$$\|v^{(r)}\|_{L^2(I)}\leq C(r,n)\Big(\|v\|^\theta_{L^2(I)}\|v^{(n)}\|^{1-\theta}_{L^2(I)}+\frac{\|v\|_{L^2(I)}}{|I|^r}\Big),$$
where $\theta=\frac{n-r}{n}$ (see for instance \cite[Theorem 7.41]{leoni}). 
Applying this inequality with $v=u^\prime$ and $n=k-1$, for all $\ell\in\{2,\dots, k-1\}$  we have that  
\begin{equation}\label{interp1}
\|u^{(\ell)}\|^2_{L^2(I)}\leq R_k\|u^\prime\|^{2\theta}_{L^2(I)}\|u^{(k)}\|^{2(1-\theta)}_{L^2(I)}+R_k \frac{\|u^\prime\|^2_{L^2(I)}}{|I|^{2(\ell-1)}},
\end{equation}
where $\theta=\frac{k-\ell}{k-1}$ and $R_k=2\max\limits_{2\leq\ell\leq k-1}
\big(C(\ell-1,k-1)\big)^2$.  
Let $\alpha\in \mathbb R$, $p=\frac{1}{\theta}=\frac{k-1}{k-\ell}$ and 
$q=\frac{1}{1-\theta}=\frac{k-1}{\ell-1}$. By Young's inequality, it follows that 
\begin{eqnarray}\label{interp2}
\e^\alpha \|u^{(k)}\|^{2(1-\theta)}_{L^2(I)} \|u^\prime\|^{2\theta}_{L^2(I)} 
&\leq&\frac{1}{q} \Big(\e^\alpha \|u^{(k)}\|^{2(1-\theta)}_{L^2(I)}\Big)^q+\frac{1}{p}\Big(\|u^\prime\|^{2\theta}_{L^2(I)}\Big)^p\nonumber\\
&\leq& \e^{\alpha \frac{k-1}{\ell-1}} \|u^{(k)}\|^{2}_{L^2(I)}+\|u^\prime\|^{2}_{L^2(I)}. 
\end{eqnarray}
By multiplying \eqref{interp1} by $\e^{\gamma_\ell}$, and applying estimate \eqref{interp2} with the choice $\alpha=\gamma_\ell=\frac{2k-1}{k-1}(\ell-1)$, we obtain \eqref{interpolation}, concluding the proof.  
\end{proof} 

We can now prove Lemma \ref{lowerlemma}. 

\begin{proof}[Proof of Lemma \rm\ref{lowerlemma}] 
Let $N\geq 1$ be fixed, and 
let $I$ satisfy hypotheses (i) and (ii). 
For any $\ell\in\{2,\dots, k-1\},$ we set 
$$I_\e^{\ell}=\Big\{t\in I: |u_\e^{(\ell)}(t)|\geq \frac{1}{N\e^\ell}\Big\}.$$
Since (ii) holds, we deduce that there exists $\ell\in\{2,\dots, k-1\}$ such that 
$|I^{\ell}_\e|\geq\frac{|I|}{k+1}$. 
Indeed, otherwise the measure of the set where at least one of the derivatives is greater than the threshold would be strictly less than $|I|$, contradicting the hypothesis (ii).  
We apply \eqref{interpolation} for this value of $\ell$ to the interval $I$ and to the function $u_\e$; since $F_\e(u_\e; I)\leq S$ and $\|u_\e\|^2_{L^2(I)}<\frac{|I|}{\e}$, we obtain 
\begin{equation*}
\frac{|I|}{k} \frac{1}{N^2\e^{2\ell}}\leq \|u_\e^{(\ell)}\|^2_{L^2(I)}\leq R_k\Big(S\e^{-\gamma_{\ell}}+\frac{1}{|I|^{2(\ell-1)}} \frac{|I|}{\e}\Big), 
\end{equation*}
with $\gamma_{\ell}=\frac{2k-1}{k-1}(\ell-1)$. 
By this inequality, we deduce that either 
\begin{equation*} 
\frac{|I|}{k} \frac{1}{N^2\e^{2\ell}}\leq 2 R_k S\e^{-\gamma_{\ell}}, 
\end{equation*}
which implies 
\begin{equation*} 
|I|\leq 2 R_k k S N^2\e^{2\ell-\gamma_{\ell}}=
2 R_k k S N^2\e^{\frac{2k- \ell-1}{k-1}}\leq 2 R_k k S N^2\e^{\frac{k}{k-1}}.
\end{equation*} 
or  
\begin{equation*} 
\frac{|I|}{k} \frac{1}{N^2\e^{2\ell}}\leq  2 R_k \frac{1}{|I|^{2(\ell-1)}} \frac{|I|}{\e}. 
\end{equation*} 
In this last case, we deduce that $|I|^{2(\ell-1)}\leq  2 R_k k N^2 \e^{2\ell-1}$, which gives
\begin{equation*} 
|I|\leq  (2 R_k k N^2)^{\frac{1}{{2(\ell-1)}}} \e^{\frac{2\ell-1}{2(\ell-1)}}\leq (2 R_k k N^2)^{\frac{1}{{2(\ell-1)}}} \e^{\frac{2k-3}{2k-4}} 
\end{equation*}  
since $\e\in(0,1)$. 
If we set, with a slight abuse of notation, 
$$R_k(N)= \max\big\{2 R_k k S N^2, \max_{2\leq \ell\leq k-1}\{(2 R_k k N^2)^{\frac{1}{{2(\ell-1)}}}\}\big\},$$ 
we then get the claim. 
\end{proof} 
For each $j\in \mathcal I_\e(N)$, let $I_j=(a_j,b_j)$ and define 
\begin{equation}\label{ajn}
\left.
\begin{array}{ll}
\vspace{2mm}
a_j^N=\inf\big\{t\in [a_j,b_j): |u_\e^{(\ell)}(t)|\leq \frac{1}{N\e^\ell} \ \hbox{\rm for all } \ell\in\{2,\ldots,k-1\}\big\},\\
b_j^N=\sup\big\{t\in (a_j,b_j]: |u_\e^{(\ell)}(t)|\leq \frac{1}{N\e^\ell} \ \hbox{\rm for all } \ell\in\{2,\ldots,k-1\}\big\}; 
\end{array}
\right. 
\end{equation} 
that is, $(a_j^N,b_j^N)$ is the maximal subinterval of $I_j$ such that  
the inequality $|u_\e^{(\ell)}|\leq \frac{1}{N\e^\ell}$  
holds at the endpoints. 
Note that by Lemma \ref{lowerlemma} we have the estimates 
\begin{equation}\label{lowerbordo} 
|(a_j,a_j^N)|\leq R_k(N) \e^{\frac{2k-3}{2k-4}}\ \ \ \hbox{\rm and }\ \ |(b_j^N,b_j)|\leq  R_k(N) \e^{\frac{2k-3}{2k-4}}. 
\end{equation}

Let $\mathcal A_\e(N)$ denote the union of these subintevals; that is, 
\begin{equation}\label{defaen}
\mathcal A_\e(N)=\bigcup_{j\in \mathcal I_\e(N)}(a_j^N, b_j^N). 
\end{equation}
In the next two statements, we will prove that in the complement of $\mathcal A_\e(N)$ the intervals where $u_\e^\prime$ is below the threshold are asymptotically negligible with respect to the rest. 
\begin{lemma}\label{osck-lemma}  
Let $N$ be such that $2N^2S\geq 1$. 
Let $I\subseteq (0,1)$ be an open interval such that 
\begin{enumerate}
\item[{\rm (i)}] $|I|\leq\frac{1}{2N^2S}\e$; 
\item[{\rm (ii)}] $I\cap 
\mathcal A_\e(N)=\emptyset$. 
\end{enumerate}
Then, 
$$\#\{j\in\mathcal I_\e: I_j\cap I\neq \emptyset\}\leq k.$$
\end{lemma} 
\begin{proof} 
Set $I=(a,b)$ and $I_j=(a_j,b_j)$ for all $j\in\mathcal I_\e$.  
Note that if $I\cap I_j\neq\emptyset$ and $I_j\not\subseteq I$, then either $I_j\cap I=(a,b_j)$ or $I_j\cap I=(a_j,b)$. Since $I\cap 
\mathcal A_\e(N)=\emptyset$, this holds for all $j\in\mathcal I_\e(N)$ such that $I\cap I_j\not=\emptyset$. 
To prove the claim, it is then sufficient to show that 
$$\#\{j\in\mathcal I_\e\setminus\mathcal I_\e(N): I_j\subset I\neq \emptyset\}\leq k-2.$$

By contradiction, suppose that there exists a family $\{j_1,\dots, j_{k-1}\}$ of $k-1$ indices in $\mathcal I_\e\setminus\mathcal I_\e(N)$ such that 
$I_{j_n}\subset I$ for all $n\in\{1,\dots, k-1\}$ and the families 
$a_{j_n}$ and $b_{j_n}$ are strictly increasing with $n$. 
This implies 
the existence of at least $k-2$ points in $I$ such that $u_\e^{\prime\prime}=0$. Indeed, let $n<k-1$. Either $u_\e^{\prime}(a_{j_n})=
u_\e^{\prime}(b_{j_n})$, or $u_\e^{\prime}(b_{j_n})=u_\e^{\prime}(a_{j_{n+1}})$; in both cases, $u_\e^{\prime\prime}$ vanishes at some $t_n\in(a_{j_n}, a_{j_{n+1}})$. 
An iteration of the application of the Rolle Theorem allows us to deduce that 
for any $\ell\in\{2,\dots, k-1\}$ there exists a point $x_\e^\ell\in I$ such that $u_\e^{(\ell)}(x_\e^\ell)=0$. 
By the Fundamental Theorem of Calculus we then obtain the estimate 
\begin{equation}\label{derk}
\|u_\e^{(\ell)}\|_{L^2(I)}\leq |I|^{k-\ell} \|u_\e^{(k)}\|_{L^2(I)}
\end{equation} 
for any $\ell\in\{2,\dots, k-1\}$. 
  
On the other hand, since for all $n$ we have that $|\{t\in (a_{j_n}, b_{j_n}): |u_\e^{(\ell)}(t)|<\frac{1}{N\e^\ell}\ \hbox{\rm for all } \ell\in\{2,\dots, k-1\}\}|=0$, there exist an integer $\ell\in\{2,\dots, k-1\}$ and a point 
$y_\e^\ell\in I$ satisfying $|u_\e^{(\ell)}(y_\e^\ell)|\geq \frac{1}{N\e^\ell}$. 
Hence, 
\begin{equation*}
|u_\e^{(\ell)}(x_\e^\ell)-u_\e^{(\ell)}(y_\e^\ell)|\geq \frac{1}{N\e^\ell}. 
\end{equation*}
Note that the function $u_\e^{(\ell)}$ is H\"older-continuous with exponent $\frac{1}{2}$ and 
constant $C_\e=\|u_\e^{(\ell+1)}\|_{L^2(I)}$; that is,  
\begin{equation}\label{holder}
|u_\e^{(\ell)}(t)-u_\e^{(\ell)}(s)|\leq \|u_\e^{(\ell+1)}\|_{L^2(I)} \sqrt{|t-s|} 
\end{equation}
for all $t,s\in I$.  
By using \eqref{derk}  in \eqref{holder} applied with $t=x_\e^\ell$ and $s=y_\e^\ell$ we then deduce 
\begin{equation*} 
\frac{1}{N\e^\ell}\leq |u_\e^{(\ell)}(x_\e^\ell)-u_\e^{(\ell)}(y_\e^\ell)|\leq |I|^{k-(\ell+1)}\|u_\e^{(k)}\|_{L^2(I)} \sqrt{|x_\e^\ell-y_\e^\ell|}.
\end{equation*}
By multiplying by $\e^\ell$, recalling that $\e^{2k-1}\|u_\e^{(k)}\|^2_{L^2(I)}\leq F_\e(u_\e;I)\leq S$ and using that $|x_\e^\ell-y_\e^\ell|\le |I|$, we obtain the estimate 
\begin{equation*} 
\frac{1}{N}\leq |I|^{k-\ell-\frac{1}{2}} \e^{\ell-k+\frac{1}{2}} \sqrt{S}. 
\end{equation*} 
The hypothesis $|I|\leq \frac{1}{2N^2 S}\e$ on the length of $I$  then gives 
\begin{equation*} 
\frac{1}{N}\leq \Big(\frac{\e}{2N^2S}\Big)^{k-\ell-\frac{1}{2}} \e^{\ell-k+\frac{1}{2}} \sqrt{S}=\Big(\frac{1}{2N^2S}\Big)^{k-\ell-\frac{1}{2}}\sqrt{S}\leq 
\frac{1}{\sqrt 2 N},  
\end{equation*} 
since by the condition $2N^2S\geq 1$ we have that $(\frac{1}{2N^2S})^{k-\ell-\frac{1}{2}}\leq(\frac{1}{2N^2S})^{k-(k-1)-\frac{1}{2}}$. 
This is a contradiction, proving the claim. 
\end{proof}

\begin{proposition}[Bound for the relative measure of intervals below threshold]\label{stimatausigma1} 
Let $N\geq 1$ be a fixed integer such that $2N^2S\geq 1$. Then, for all $\e\in(0,1)$ and for  any interval  
$I$ such that 
\begin{equation}\label{noder} 
I\cap \mathcal A_\e(N)=\emptyset,  
\end{equation} 
the following estimate holds
\begin{equation*}
|I\cap \mathcal A_\e|\leq kR_k(N)\e^{\frac{1}{2(k-2)}}
(2N^2S|I|+\e). 
\end{equation*}
\end{proposition} 
\begin{proof} 
Note that we can apply Lemma \ref{lowerlemma} to $I_j$ with $j\in\mathcal I_\e\setminus \mathcal I_\e(N)$, and to the intervals $(a_j,a_j^N)$ and $(b_j^N,b_j)$ with $j\in\mathcal I_\e(N)$ as in \eqref{lowerbordo}, 
obtaining that for each $I_j$ such that $I\cap I_j\neq\emptyset$ the estimate 
\begin{equation}\label{stimaintersezione}
|I_j\cap I |\leq R_k(N)\e^{\frac{2k-3}{2k-4}}
\end{equation}
holds. 

\smallskip

We now consider separately the cases $|I|\leq \frac{1}{2N^2S}\e$ and $|I|>\frac{1}{2N^2S}
\e$. 

\smallskip 

\noindent{\em Case $|I|\leq \frac{1}{2N^2S}\e$.} 
Applying Lemma \ref{osck-lemma} we obtain that  the number of $j\in\mathcal I_\e$ such that 
$I\cap I_j\neq\emptyset$ is at most $k$. 
Hence, by \eqref{stimaintersezione} we obtain the claim, since 
\begin{equation*}
|I \cap \mathcal A_\e 
|\leq k R_k(N)\e^{\frac{2k-3}{2k-4}}. 
\end{equation*}

\smallskip 

\noindent{\em Case $|I|>\frac{1}{2N^2S}\e$.} 
We can subdivide $I$ in $M=\lfloor \frac{2N^2 S|I|}{\e}\rfloor+1$ disjoint open subintervals, each one with length equal to $\frac{|I|}{M}\leq \frac{1}{2N^2S}\e$.  
Applying Lemma \ref{lowerlemma} and \eqref{stimaintersezione} to each such interval and summing up, we obtain  
\begin{eqnarray*}
|I \cap \mathcal A_\e 
|&\leq& M k R_k(N)\e^{\frac{2k-3}{2k-4}} 
\ \leq \ \Big(\frac{2N^2 S|I|}{\e} +1\Big)k R_k(N)\e^{\frac{2k-3}{2k-4}}\\
&\leq& 
(2N^2 S|I| +\e)k R_k(N)\e^{\frac{1}{2(k-2)}},  
\end{eqnarray*} 
which is the claim. 
\end{proof}
\begin{remark}\label{epsS}\rm  
By Proposition \ref{stimatausigma1}, we deduce that if $N$ satisfies  $2N^2S\geq 1$ and $I$ is an interval such that \eqref{noder} holds, then 
\begin{equation}\label{epsSeq}
|I|\leq 4S\e  
\end{equation} 
for $\e$ small enough. 
Indeed, 
$|I\setminus \mathcal A_\e 
|\leq \e F_\e(u_\e)\leq \e S$, 
and applying Proposition \ref{stimatausigma1} we deduce 
\begin{equation}\label{parziale}
|I|\Big(1-2N^2S kR_k(N)\e^{\frac{1}{2(k-2)}}\Big)\leq kR_k(N)\e^{\frac{1}{2(k-2)}+1}+\e S. 
\end{equation} 
For any $\e>0$ such that 
$$\e<\widetilde\e_k(N):= 
\Big(\frac{1}{kR_k(N)}\min\Big\{\frac{1}{4N^2S}, S\Big\}\Big)^{2(k-2)},$$ 
 estimate \eqref{parziale} implies $
\frac{1}{2}|I|\leq 2\e S$, 
and the claim. 
\end{remark}

We can now asymptotically estimate the relative measure of 
$\mathcal A_\e$ in $I$, if the interval $I$ satisfying \eqref{noder} is ``large enough''.  
More precisely, by Proposition \ref{stimatausigma1} we deduce the following asymptotic result. 
\begin{proposition}\label{asym} 
Let $N$ be such that $2N^2S\geq 1$.   
Let $\{I(\e)\}$ be a family of intervals such that
\begin{equation}\label{ipotesiinf}
\lim_{\e\to 0}\frac{\e^{1+\frac{1}{2(k-2)}}}{|I(\e)|}=0,
\end{equation}
and $I(\e)\cap \mathcal A_\e(N) 
=\emptyset$,  
then 
\begin{equation*}
\lim_{\e\to 0}\frac{|I(\e) \cap \mathcal A_\e |}{|I(\e)|}=0. 
\end{equation*}
\end{proposition}
\begin{proof}
We can apply Proposition \ref{stimatausigma1} to each $I(\e)$, obtaining 
\begin{equation*}
\frac{|I(\e) \cap \mathcal A_\e|}{|I(\e)|}
\leq kR_k(N)\e^{\frac{1}{2(k-2)}}
\Big(2N^2S+\frac{\e}{|I(\e)|}\Big).  
\end{equation*} 
By \eqref{ipotesiinf}, we get the claim.  
\end{proof}

\subsection{Finer analysis of `small' intervals}\label{smallint}
In the proof of the coerciveness, we will also need to asymptotically estimate this proportion in the case of  a family of ``small'' intervals $I(\e)$; that is, intervals not satisfying 
\eqref{ipotesiinf}. To this end, we introduce a further threshold, as follows. 

We fix $r\in (0,1)$ and define the set of indices 
$\mathcal I_\e^r$ by setting 
\begin{equation}\label{intdelta}
\mathcal I_\e^r=\Big\{j\in \mathcal I_\e: \int_{I_j}(u^\prime_\e)^2\, dt<\frac{r}{\e}|I_j|\Big\}, 
\end{equation}
and define 
\begin{equation}\label{defaer}
\mathcal A_\e^r=\bigcup_{j\in \mathcal I_\e^r}I_j, 
\end{equation}
noting that by the equiboundedness of $F_\e(u_\e)$ we deduce 
\begin{equation}\label{stimalunghr}
|(0,1)\setminus\mathcal A_\e^r|\leq\frac{\e S}{r}. 
\end{equation}
We now give of a lower bound for the length of the intervals $I_j$ with $j\in  \mathcal I_\e^r$ for $r$ small enough; that is, intervals where the $L^2$-norm of the derivative of $u_\e$ is small enough.  
The proof is again based on Lemma \ref{lemmainterp}. 

\begin{lemma} 
\label{upperlemma}  
There exist a threshold $r_k\in (0,1)
$ 
and a constant $\widetilde C_k>0$ 
such that 
for any $\e\in(0,1)$ and for any interval $I$ satisfying 
\begin{enumerate}
\item[{\rm (i)}] $|u_\e^\prime|<\frac{1}{\sqrt\e}$ in $I$ and $|u_\e^\prime|=\frac{1}{\sqrt\e}$ at at least one of the endpoints; 
\item[{\rm (ii)}] $\int_I (u^\prime_\e)^2\, dt<\frac{r_k}{\e}|I|$ 
\end{enumerate}
the following estimate holds
\begin{equation}\label{stimaI}
|I|\geq 
\widetilde C_k \e^{\frac{k}{k-1}}. 
\end{equation} 
In particular, \eqref{stimaI} holds for all $\e\in(0,1)$ and $I_j$ with $j\in  \mathcal I_\e^{r_k}$. 
\end{lemma}

\begin{proof} 
Let $r\in(0,1)$ be fixed, and let $I=(a,b)$ be an interval as in the hypothesis, with $r$ in the place of $r_k$ and with $|u^\prime_\e(a)|=\frac{1}{\sqrt\e}$ 
(the case when $|u^\prime_\e(b)|=\frac{1}{\sqrt\e}$ can be treated exactly in the same way). 
Then, there exists $t_\e\in (a,b)$ (also depending on $r$) such that $|u^\prime_\e(t_\e)|<\frac{\sqrt{r}}{\sqrt\e}$. 
It follows that 
\begin{equation}\label{stimaint1}
\int_{I}(u_\e^{\prime\prime})^2\, dt\geq \int_{a}^{t_\e}   (u_\e^{\prime\prime})^2\, dt \geq  \frac{1}{|t_\e-a|}\Big(\int_{a}^{t_\e} u_\e^{\prime\prime}\, dt\Big)^2
\geq\frac{(1-\sqrt{r})^2}{\e (b-a)}. 
\end{equation}
Now, by using \eqref{stimaint1}  and the interpolation inequality \eqref{interpolation} with $\ell=2$, we can estimate 
\begin{equation}\label{stimasomma}
\frac{\e^{\frac{2k-1}{k-1}}(1-\sqrt{r})^2}{\e|I|}\leq R_k S+R_k\frac{\e^{\frac{2k-1}{k-1}}}{|I|^2}\frac{r}{\e}|I| 
\end{equation}
since $\int_{I}(u^\prime_\e)^2\, dt< \frac{r}{\e}|I|$ and $F_\e(u_\e; I)\leq S$. 
Then, 
$$\hbox{\rm either } \ \  \frac{\e^{\frac{2k-1}{k-1}}(1-\sqrt{r})^2}{\e|I|}\leq 2 R_kS \ \ \hbox{\rm or }\  \ \frac{\e^{\frac{2k-1}{k-1}}(1-\sqrt{r})^2}{\e|I|}\leq 2R_k\frac{\e^{\frac{2k-1}{k-1}}}{\e |I|}r.$$ 
In the second case, we have that $(1-\sqrt{r})^2\leq 2R_k r$; then, choosing $r_k\in (0,1)$ small enough to have  that $(1-\sqrt{r_k})^2>2 R_k r_k$, we obtain that for any $\e\in(0,1)$ the first inequality holds true; that is, 
$$ |I| 
\geq \frac{(1-\sqrt{r_k})^2}{2 R_kS}\e^{\frac{k}{k-1}}, 
$$ 
which is the claim with $\widetilde C_k=\frac{(1-\sqrt{r_k})^2}{2 R_kS}$. 
\end{proof}

We now prove that, except for a set with total measure of order strictly less than $\e$, 
each interval $I_j$ with $j\in\mathcal I_\e^{r_k}$ contains a set with positive measure where 
$u_\e^{(\ell)}<\frac{1}{N\e^\ell}$ for 
all the derivatives up to the order $\ell(k)$ defined by 
\begin{equation}\label{defellk}
\ell(k)=\max \Big(\mathbb N\cap \Big[1,\frac{k+1}{2}\Big)\Big), 
\end{equation}
and for any $N\geq 1$, we define $\mathcal I_\e^\ast(N)$ as the set of indices $j\in \mathcal I_\e^{r_k}$ such that 
\begin{equation}\label{definast} 
\Big|\Big\{t\in I_j: |u_\e^{(\ell)}(t)|<\frac{1}{N\e^\ell} \ \hbox{\rm for all } \ell\in\{1,\dots, \ell(k)\}\Big\}\Big|>0. 
\end{equation} 
Note that, since $|u_\e'(t)|<\frac{1}{\sqrt\e}$ in each $I_j$, 
then if $\e<\frac{1}{N^2}$ the inequality for $\ell=1$ is true for all $t\in I_j$. 
\begin{lemma}\label{lemmar}  
Let $N\geq 1$ be fixed. 
There exists a threshold $\e_k(N)\in (0,1)$ 
such that 
\begin{equation}\label{eqr}
\sum_{j\in \mathcal I_\e^{r_k}\setminus \mathcal I_\e^\ast(N)}|I_j|\leq R_k k^2 S N^2\e^{\frac{k}{k-1}}
\end{equation}
for all $\e\in (0,\e_k(N))$, 
where $r_k$ is the constant given in Lemma~\ref{upperlemma}. 
\end{lemma}
\begin{proof}  
If $k=3,$ the set of indices $\mathcal I_\e^\ast(N)$ coincides with the whole $\mathcal I_\e$ for all $\e<\e_3(N)=\frac{1}{N^2}$, and the claim holds since $\mathcal I_\e^r\setminus\mathcal I_\e^\ast(N)=\emptyset$ for any $r\in(0,1)$.     

\smallskip 
We have to prove the claim for $k>3$. In this case, $\ell(k)\geq 2$. 
For each integer $\ell\in \{1,\dots, \ell(k)\}$, let  
$$I^{\ell}_j=\Big\{t\in I_j: |u_\e^{(\ell)}(t)|\geq \frac{1}{N\e^\ell}\Big\}$$ 
as defined in Lemma \ref{lowerlemma}, and 
let $\mathcal I_\e^{r_k}(\ell)$ be the set of indices 
$j\in\mathcal I_\e^{r_k}\setminus\mathcal I_\e^\ast(N)$ such that $|I^{\ell}_j|\geq\frac{|I_j|}{k}$.
Now, for any $\e\in (0,\frac{1}{N^2})$, we have that 
$I^{1}_j=\emptyset$ for all $j$, and consequently $\mathcal I_\e^{r_k}(1)=\emptyset$.  
We then have  
\begin{equation}\label{unione}
\mathcal I_\e^{r_k}\setminus\mathcal I_\e^\ast(N)=\bigcup_{\ell=2}^{\ell(k)}\mathcal I_\e^{r_k}(\ell). 
\end{equation}
Hence, if $j\in\mathcal I_\e^{r_k}\setminus\mathcal I_\e^\ast(N)$ 
we have that 
$\sum_{\ell=2}^{\ell(k)} |I^{\ell}_j|\ge|I_j|$. 
Since $\ell(k)<k$, there exists an integer $\ell\in \{2,\dots, \ell(k)\}$ such that $|I^{\ell}_j|\geq\frac{1}{k}|I_j|$. 
As in the proof of Lemma \ref{lowerlemma}, by applying interpolation inequality 
\eqref{interpolation} to each interval $I_j$ with $j\in \mathcal I_\e^{r_k}(\ell)$ and to the function $u_\e$, 
we obtain 
\begin{equation*}
\frac{|I_j|}{k} \frac{1}{N^2\e^{2\ell}}\leq \|u_\e^{(\ell)}\|^2_{L^2(I_j)}\leq R_k\Big(F_\e(u_\e; I_j)\e^{-\gamma_{\ell}}+r_k\frac{1}{|I_j|^{2(\ell-1)}} \frac{|I_j|}{\e}\Big), 
\end{equation*}
with $\gamma_{\ell}=\frac{2k-1}{k-1}(\ell-1)$. 
By this inequality, we deduce that either 
\begin{equation}\label{estok}
\frac{|I_j|}{k} \frac{1}{N^2\e^{2\ell}}\leq  2R_k F_\e(u_\e; I_j)\e^{-\gamma_{\ell}}, 
\end{equation}
or  
\begin{equation}\label{estell0}
\frac{|I_j|}{k} \frac{1}{N^2\e^{2\ell}}\leq   2R_k r_k\frac{1}{|I_j|^{2(\ell-1)}} \frac{|I_j|}{\e}.       
\end{equation} 
Suppose now that $j$ is such that \eqref{estell0} holds. 
Since we can apply Lemma \ref{upperlemma} to each interval $I_j$ with $j\in\mathcal I_\e^{r_k}$, from \eqref{estell0} we deduce that 
\begin{equation}\label{estell}
\widetilde C_k\e^{\frac{k}{k-1}}\leq |I_j|\leq (2 k R_k N^2)^{\frac{1}{2(\ell-1)}} \e^{\frac{2\ell-1}{2(\ell-1)}}\leq R_k(N) 
\e^{\frac{2\ell-1}{2(\ell-1)}}, 
\end{equation}  
where $R_k(N)=\max_\ell\{(2 k R_k r_k N^2)^{\frac{1}{2(\ell-1)}}\}$ and we recall that $\widetilde C_k>0$ depends only on $k$ and $S$. 
Since $\ell\leq \ell(k)<\frac{k+1}{2}$, 
there exists a threshold $\e_k(N, \ell)\in (0,\frac{1}{N^2})$ such that 
$$
\e^{\frac{2\ell-1}{2(\ell-1)}}< \frac{\widetilde C_k}{R_k(N)}\e^{\frac{k}{k-1}}$$
for any $\e\in(0,\e_k(N,\ell))$, giving a contradiction.  
Hence, if $\e\in(0,\e_k(N,\ell))$ then for all $j\in \mathcal I_\e^{r_k}(\ell)$ inequality \eqref{estok} holds; by multiplying by $kN^2\e^{2\ell}$ and summing up, for any $\e\in(0,\e_k(N,\ell))$ 
we obtain 
\begin{eqnarray*}
\sum_{j\in \mathcal I_\e^{r_k}(\ell)}|I_j|&\leq&2k R_k N^2 \e^{2\ell-\gamma_{\ell}}\sum_{j\in \mathcal I_\e^{r_k}(\ell)}F_\e(u_\e; I_j)\ \leq \ 2 k R_k N^2 S
\e^{\frac{2k-1-\ell}{k-1}}\\
&\leq& 2 k R_k N^2 S
\e^{\frac{k}{k-1}}, 
\end{eqnarray*}
where the last inequality holds since in particular $\ell\leq k-1$. 
We can repeat the argument for each $\ell$ and sum up, 
obtaining, thanks to \eqref{unione}, that the estimate 
\begin{equation*}
\sum_{j\in \mathcal I_\e^{r_k}\setminus \mathcal I_\e^\ast(N)}|I_j|\leq R_k k^2 S N^2
\e^{\frac{k}{k-1}}, 
\end{equation*}
holds for $\e<\frac{1}{N^2}$ and $\e<\min\{\e_k(N,\ell): \ell\in \{2,\dots, \ell(k)\}\}$, as claimed. 
\end{proof} 
Following the construction of the intervals $(a_j^N,b_j^N)$ for $j\in\mathcal I_\e(N)$, 
we define, for any interval $I_j$ with $j\in\mathcal I_\e^\ast(N)$, the maximal subinterval 
$(a_j^{\ast,N},b_j^{\ast,N})\subset I_j$ such that $|u_\e^{(\ell)}|\leq\frac{1}{N\e^\ell}$ for all $\ell\in\{1,\dots, \ell(k)\}$ at the endpoints. 
\begin{remark}\label{bordir}\rm{The same argument used in the proof of Lemma \ref{lemmar} allows to treat the intervals $I$ which are one of the two intervals  in $I_j\setminus (a_j^{\ast,N},b_j^{\ast,N})$ for some $j$. That argument shows that the sum of the lengths of such intervals $I$ that satisfy the inequality in \eqref{intdelta} is estimated by $R_k k^2 S N^2\e^{\frac{k}{k-1}}$ for $\e$ small enough.}
\end{remark}

\section{Proof of the compactness and $\Gamma$-convergence}\label{sec:proof}
In this section, we prove Theorem \ref{main}. The main effort will be spent in proving the compactness and lower bound. To that end, the argument is to give a lower bound for $F_\e(u_\e)$ in terms of a sequence $G_\e(v_\e)$ where $v_\e$ are close to $u_\e$ and $G_\e$ are a sequence of equicoercive energies. The construction of $G_\e$ and $v_\e$ will involve the optimal-profile problems studied in Section \ref{optimalprofile}, whose use will be made possible by the analysis in Section \ref{inter}. 

\subsection{Notation for the localization on subintervals}
Let $N\geq 1$ be such that $2N^2 S\geq 1$ and let $\mathcal I_\e(N)$ be defined by \eqref{defIN}.  
For any $j\in \mathcal I_\e(N)$, let $a_j^N$ and $b_j^N$ be defined by \eqref{ajn}; that is, 
$(a_j^{N}, b_j^{N})=I_j^{N}$ 
is the maximal subinterval 
of $I_j=(a_j,b_j)$ such that $|u_\e^{(\ell)}|\leq \frac{1}{N\e^\ell}$
holds at the endpoints for all integers $\ell\in\{2,\ldots,k-1\}$.   
We define 
$$a_\e=\inf\{a_j^{N}: \ j\in\mathcal I_\e(N)\} \ \ \hbox{\rm and }\ \ 
b_\e=\sup\{b_j^{N}: \ j\in\mathcal I_\e(N)\},$$
noting that by Remark \ref{epsS} we have that 
\begin{equation}\label{aebe}
|(0,1)\setminus (a_\e,b_\e)|=O(\e)_{\e\to 0}.
\end{equation} 
We now describe the complementary set in $(a_\e,b_\e)$ of 
the set $\mathcal A_\e(N)$ given by the union of the intervals $(a_j^{N}, b_j^{N})$, as defined in \eqref{defaen}. 
To this end, for $t\in (a_\e,b_\e)\setminus 
\mathcal A_\e(N)$ we set 
\begin{eqnarray*}
\tau(t)=\sup\{b_j^{N}: b_j^{N}\leq t, \ j\in\mathcal I_\e(N)\};\\ 
\sigma(t)=\inf \{a_j^{N}: a_j^{N}\geq t, \ j\in\mathcal I_\e(N)\}, 
\end{eqnarray*} 
noting that the sets are not empty since $t\in (a_\e,b_\e)$. Note moreover that $\tau(t)$ and $\sigma(t)$ may coincide. 

From these definitions, the following properties hold for all $t\in (a_\e,b_\e)\setminus 
\mathcal A_\e(N)$: 
\begin{enumerate} 
\item[{\rm (i)}] $[\tau(t),\sigma(t)]\cap 
\mathcal A_\e(N)=\emptyset$; 
\item[\rm (ii)] if $I_j\cap [\tau(t), \sigma(t)]\neq\emptyset$ for $j\in\mathcal I_\e\setminus \mathcal I_\e(N)$, then  
$I_j\subset [\tau(t), \sigma(t)]$
\item[{\rm (iii)}] if $(a_j, a_j^{N}]\cap [\tau(t), \sigma(t)]\neq\emptyset$, then  
$(a_j, a_j^{N}]\subset[\tau(t), \sigma(t)]$; the same property holds for $[b_j^{N}, b_j)$; 
\item[{\rm (iv)}] by the continuity of the derivatives of $u_\e$,  
$$|u_\e^{\prime}(\tau(t))|\leq\frac{1}{\sqrt\e} \ \ \hbox{\rm and } \ \ |u_\e^{(\ell)}(\tau(t))|\leq \frac{1}{N\e^\ell}  \ \ \hbox{\rm for all }\ \ell\in\{2,\dots, k-1\},$$ 
and the same properties hold in $\sigma(t)$.  
\end{enumerate}  
Moreover, the intervals $(\tau(t), \sigma(t))$ either coincide or are disjoint; indeed, 
if $s\in (\tau(t), \sigma(t))$ for some $t$, then $(\tau(t), \sigma(t))=(\tau(s), \sigma(s))$.
Consider now a point $t$ such that the corresponding interval is a point; that is, 
$\tau(t)=\sigma(t)=t$. In this case, 
the common value also coincides with $\sup\{b_j: b_j\leq t, \ j\in\mathcal I_\e(N)\}$ and with 
$\inf\{a_j: a_j\geq t, \ j\in\mathcal I_\e(N)\}$. Since $|u_\e^\prime(a_j)|=|u_\e^\prime(b_j)|=\frac{1}{\sqrt\e}$, by the continuity of $u_\e^\prime$ we also have 
 $|u_\e^\prime(t)|=\frac{1}{\sqrt\e}$. 
It follows that 
 \begin{equation*}
 J(\e):=\Big\{t\in(a_\e,b_\e)\setminus 
\mathcal A_\e(N) : \tau(t)=\sigma(t)\Big\}\subset \Big\{t: |u_\e^\prime(t)| = \frac{1}{\sqrt\e}\Big\}, 
 \end{equation*}  
so that  $|J(\e)| \leq \e S$. 
Then, we can write 
the set $(a_\e,b_\e)\setminus 
\mathcal A_\e(N)$ as the union of a set with measure of order $\e$ and of a disjoint union of intervals of the form $(\tau(t), \sigma(t))$; that is, there exists a countable set of indices $\Lambda_\e(N)$ such that
$$(a_\e,b_\e)\setminus   
\mathcal A_\e(N) 
=J(\e)\cup \bigcup_{\lambda\in \Lambda_\e(N)}(\tau_\lambda, \sigma_\lambda),$$ 
where 
the intervals $(\tau_\lambda, \sigma_\lambda)$ are pairwise disjoint and the endpoints $\tau_\lambda$ and $\sigma_\lambda$
belong to the sets $\{\tau(t): t\in (a_\e,b_\e)\setminus 
\mathcal A_\e(N)\}$ and $\{\sigma(t): t\in (a_\e,b_\e)\setminus 
\mathcal A_\e(N)\}$, respectively. 

Moreover, by Remark \ref{epsS} we deduce that for $\e$ small enough; more precisely, for all $\e\in(0,\widetilde \e_k(N))$, where $\widetilde \e_k(N)$ is given in Remark \ref{epsS},  the estimate 
\begin{equation}\label{lengthtausigma} 
 \sigma_\lambda-\tau_\lambda\leq 4S\e 
\end{equation}
holds for any $\lambda\in\Lambda_\e(N)$. 

\medskip 
In the proof of the equicoerciveness and of the lower bound for the $\Gamma$-limit, 
we will use some estimates for the energy in each $(\tau_\lambda,\sigma_\lambda)$ in dependence of the value 
\begin{equation}\label{defz}
z_\e^\lambda:=u_\e(\sigma_\lambda)-u_\e(\tau_\lambda),
\end{equation} 
and  a consequent global estimate for the energy in the union of these sets.  

To obtain these estimates, we introduce for any interval $I\subset (0,1)$ the minimum problem 
\begin{equation}\label{defmue}
\mu_\e(I)=\min_v\Big\{\frac{|I|}{\e}+\e^{2k-1}\int_{I} 
(v^{(k)})^2\, dt \Big\},
\end{equation}  
where the minimum is taken over the set of functions $v\in H^{k}(I)$ satisfying the same boundary conditions of the function $u_\e$ up to order $k-1$. 
In order to provide a lower estimate for $\mu_\e(\tau_\lambda,\sigma_\lambda)$ by changing the variable in the minimum problem, we will choose a scaling parameter $\beta_\e>0$ in dependence of the value of $|z_\e^\lambda|$, and for any $v$ admissible test function for the problem \eqref{defmue} with $I=(\tau_\lambda,\sigma_\lambda)$, we will define the scaled function $w$,  by setting 
\begin{equation}\label{scaling}
v(t)=z_\e^\lambda w\Big(\frac{t}{\beta_\e}-\frac{\tau_\lambda+\sigma_\lambda}{2\beta_\e}\Big),
\end{equation}
with domain $(-\frac{\sigma_\lambda-\tau_\lambda}{2\beta_\e},\frac{\sigma_\lambda-\tau_\lambda}{2\beta_\e})$. 

\subsection{Localization on intervals with large jumps}\label{sec42}
In order to give a lower bound for $F_\e(u_\e)$, it is convenient to introduce a further parameter $\theta\in(0,1)$ and separately consider the family of indices $\lambda\in \Lambda_\e(N)$ such that $|z_\e^\lambda|\geq \theta$ and 
$|z_\e^\lambda|<\theta$, respectively.

The first estimate concerns the energy on intervals $(\tau_\lambda,\sigma_\lambda)$ such that $|u_\e(\tau_\lambda)-u_\e(\sigma_\lambda)|\geq \theta$. Let $\Lambda_\e^{+}(N,\theta)$ denote the set of such indices in $\Lambda_\e(N)$.

We start by proving a lower bound for the length of the intervals $(\tau_\lambda, \sigma_\lambda)$ such that $\lambda\in\Lambda_\e^{+}(N,\theta)$. 

\begin{lemma}\label{leng}
There exists $N(\theta)$ such that for all 
$N\geq N(\theta)$, $\e\in(0,1)$ and $\lambda\in\Lambda_\e^{+}(N,\theta)$, we have   
\begin{equation}\label{mess}
\sigma_\lambda-\tau_\lambda
\geq\frac{1}{2N^2S}\e.
\end{equation} 
\end{lemma}

\begin{proof}
For all $\lambda\in\Lambda_\e^{+}(N,\theta)$ we define the scaled function 
$v_{\e}^{\lambda}$ such that $u_\e(t)=z_\e^\lambda v^\lambda_\e(\frac{t-\tau_\lambda}{\sigma_\lambda-\tau_\lambda})$, obtaining  
\begin{eqnarray}\label{distat}\nonumber
S\geq F_\e(u_\e;(\tau_\lambda,\sigma_\lambda))&\geq&\frac{\e^{2k-1}(z_\e^\lambda)^2}{(\sigma_\lambda-\tau_\lambda)^{2k-1}}\int_{0}^1 ((v^\lambda_\e)^{(k)})^2\, ds\\
&\geq&\frac{\e^{2k-1}(z_\e^\lambda)^2}{(\sigma_\lambda-\tau_\lambda)^{2k-1}}m_\theta(N),
\end{eqnarray}
where 
\begin{eqnarray*}
&&m_\theta(N):=\min\Big\{\int_0^1 (v^{(k)})^2\, ds: v\in H^k(0,1), v(1)-v(0)=1, \\
&&\hspace{2cm}|v^{(\ell)}(0)|\leq\frac{(4S)^\ell}{\theta N}, |v^{(\ell)}(1)|\leq\frac{(4S)^\ell}{\theta N} 
\ \hbox{\rm for all } \ell\in\{1,\dots, k-1\}\Big\}.  
\end{eqnarray*} 
The value $\theta$ appears in this minimum problem after scaling the boundary condition and using that $|z_\e^\lambda|\geq\theta$. Note also that $\sigma_\lambda-\tau_\lambda\leq 4S\e$ by \eqref{lengthtausigma}.

Note that $m_\theta(N)\to+\infty$ as $N\to+\infty$. Then, there exists a positive integer $N(\theta)$ such that $m_\theta(N)\geq 1$ and 
$(\frac{\theta^2}{S})^{\frac{1}{2k-1}}\geq\frac{1}{2N^2 S}$ for $N\geq N(\theta)$. 
We choose $N(\theta)$ also satisfying $2N(\theta)^2S\geq 1$. From \eqref{distat}
 it follows that for all $N\geq N(\theta)$ and $\e\in(0,1)$   
$$
\sigma_\lambda-\tau_\lambda
\geq \e \Big(\frac{\theta^2}{S}\Big)^{\frac{1}{2k-1}}\geq\frac{1}{2N^2S}\e, 
$$
concluding the proof. 
\end{proof}

\begin{remark}\label{lenggen}\rm 
Note that more in general Lemma \ref{leng} holds for intervals $(\tau,\sigma)$ 
such that $|u_\e^{(\ell)}|\leq \frac{1}{N\e^\ell}$ at $\tau$ and $\sigma$ for $\ell\in\{1,\ldots,\ell(k)\}$, with $\ell(k)$ defined in \eqref{defellk}. 
The proof is the same, except that in the definition of $m_\theta(N)$ we only consider boundary conditions up to order $\ell(k)$, which is sufficient to have $m_\theta(N)\to+\infty$ as $N\to+\infty$.
\end{remark}

\begin{proposition}[Lower estimate in the case $|z_\e^\lambda|\geq\theta$]\label{propsopra} 
Let $\theta\in(0,1)$ be fixed. 
There exists an integer $N(\theta)$ such that for all 
$N\geq N(\theta)$ and 
for all $\e\in(0,\frac{1}{N^2})$,  
the following estimate holds: 
\begin{equation}\label{stima1fefin} 
\sum_{\lambda\in\Lambda_\e^+(N,\theta)}\!F_\e(u_\e;(\tau_\lambda,\sigma_\lambda))\geq  C^{N}_\e\, m_k(N\theta)\!\!\sum_{\lambda\in\Lambda_\e^+(N,\theta)}\!\!|u_\e(\sigma_\lambda)-u_\e(\tau_\lambda)|^{\frac{1}{k}}, 
\end{equation}
where $C^{N}_\e=1-
4kR_k(N)N^2S\e^{\frac{1}{2(k-2)}}$ and  $m_k(\cdot)$ is the minimum defined in \eqref{def-MkN}. 
\end{proposition}

\begin{proof}
We start by proving that, if $\lambda\in\Lambda_\e^{+}(N,\theta)$, the measure of the intersection between $(\tau_\lambda, \sigma_\lambda)$ and 
the set $\mathcal A_\e$ of the points where $|u_\e^\prime|<\frac{1}{\sqrt\e}$ 
is asymptotically negligible with respect to $|(\tau_\lambda, \sigma_\lambda)|$. 
Indeed, 
by Proposition \ref{stimatausigma1} and by using estimate \eqref{mess} 
we deduce that  
\begin{eqnarray*}
\frac{|(\tau_\lambda, \sigma_\lambda) \cap 
\mathcal A_\e|}{|(\tau_\lambda, \sigma_\lambda)|}
&\leq& kR_k(N)\e^{\frac{1}{2(k-2)}}
\Big(2N^2S+\frac{\e}{|(\tau_\lambda, \sigma_\lambda)|}\Big) \\
&\leq& 
4kR_k(N)N^2S\e^{\frac{1}{2(k-2)}}
\end{eqnarray*} 
for all $N\geq N(\theta)$ and for all $\e\in(0,1)$. 
Hence, 
\begin{eqnarray*}
\sum_{\lambda\in\Lambda_\e^{+}(N,\theta)}\int_{\tau_\lambda}^{\sigma_\lambda} f_\e(u_\e^\prime)\, dt 
&\geq& \sum_{\lambda\in\Lambda_\e^{+}(N,\theta)}
\frac{1}{\e}\Big|(\tau_\lambda,\sigma_\lambda)\cap\Big\{t: |u_\e^\prime|\geq\frac{1}{\sqrt\e}\Big\}\Big|\\
&\geq&(1-4kR_k(N)N^2S\e^{\frac{1}{2(k-2)}})\sum_{\lambda\in\Lambda_\e^{+}(N,\theta)}
\frac{\sigma_\lambda-\tau_\lambda}{\e}. 
\end{eqnarray*}
It follows that 
\begin{eqnarray}\label{stima1fe}
\sum_{\lambda\in\Lambda_\e^{+}(N,\theta)} \!\!F_\e(u_\e;(\tau_\lambda,\sigma_\lambda))&\geq &C^{N}_\e
 \sum_{\lambda\in\Lambda_\e^{+}(N,\theta)}\!\!\Big(\frac{\sigma_\lambda-\tau_\lambda}{\e}+\e^{2k-1}\int_{\tau_\lambda}^{\sigma_\lambda}\!\!(u_\e^{(k)})^2\, dt\Big)
\nonumber\\
&\geq&C^{N}_\e\sum_{\lambda\in\Lambda_\e^{+}(N,\theta)} \mu_\e(\tau_\lambda, \sigma_\lambda),   
\end{eqnarray} 
where $\mu_\e(\tau_\lambda, \sigma_\lambda)$ is the minimum defined in \eqref{defmue}.  

For any $v$ admissible test function for $\mu_\e(\tau_\lambda, \sigma_\lambda)$, we define a scaled function $w$ as in \eqref{scaling} by choosing the scaling factor $\beta_\e=\e|z_\e^\lambda|^{\frac{1}{k}}$; that is, 
$$v(t)=z_\e^\lambda w\Big(\frac{2t-(\tau_\lambda+\sigma_\lambda)}{2\e|z_\e^\lambda|^{\frac{1}{k}}}\Big).$$ 
Letting $T_\e=\frac{\sigma_\lambda-\tau_\lambda}{\beta_\e}$, we obtain 
\begin{equation*}
\frac{\sigma_\lambda-\tau_\lambda}{\e}+\e^{2k-1}\int_{\tau_\lambda}^{\sigma_\lambda} 
(v^{(k)})^2\, dt=|z_\e^\lambda|^{\frac{1}{k}}\Big(T_\e+\int_{-\frac{T_\e}{2}}^{\frac{T_\e}{2}} (w^{(k)})^2\, dt\Big), 
\end{equation*}
with $w$ satisfying the boundary conditions $w(\frac{T_\e}{2})-w(-\frac{T_\e}{2})=1$ and 
\begin{eqnarray}\label{condiz}
&&\Big|w^\prime\Big(\pm\frac{T_\e}{2}\Big)\Big|
\leq\sqrt\e|z_\e^\lambda|^{\frac{1}{k}-1}
\leq\frac{\sqrt{\e}}{\theta}\nonumber\\ 
&&\Big|w^{(\ell)}\Big(\pm\frac{T_\e}{2}\Big)\Big|
\leq \frac{|z_\e^\lambda|^{\frac{\ell}{k}-1}}{N } 
\leq \frac{1}{\theta N} 
\ \ \hbox{\rm 
for all } \ \ell\in\{2,\dots, k-1\}. 
\end{eqnarray} 
For all $\e\in(0,\frac{1}{N^2})$ and for all $\lambda\in\Lambda^{+}_\e(N,\theta)$ we then have  
\begin{equation}\label{stimamin1}
\mu_\e(\tau_\lambda, \sigma_\lambda)\geq m_k(N\theta)|u_\e(\sigma_\lambda)-u_\e(\tau_\lambda)|^{\frac{1}{k}}. 
\end{equation}
By using this inequality in \eqref{stima1fe}, we obtain the claim. 
\end{proof}

\subsection{Localization on intervals with small jumps}
Before stating a lower estimate in the case $|z_\e^\lambda|<\theta$, we have to make some preliminary remarks. In this case, we do not have a general  estimate for the relative length of $(\tau_\lambda, \sigma_\lambda) \cap 
\mathcal A_\e$ 
in $(\tau_\lambda, \sigma_\lambda)$. Moreover, if $|z_\e^\lambda|<\theta$ 
the scaling 
$\beta_\e=\e|z_\e^\lambda|^{\frac{1}{k}}$ in the estimate for the minimum problems does not imply that the derivatives at the endpoints are less than $\frac{1}{\theta N}$ as in \eqref{condiz}. 
We will have then to consider different scaling factors. 

Given $\theta\in(0,1)$, we let $N\geq N(\theta)$ and $\e\in(0,\frac{1}{N^2})$ as above. 
Let $\Lambda_\e^-(N,\theta)$ be the set of the indices $\lambda\in\Lambda_\e(N)$ such that 
$|z_\e^\lambda|<\theta$. 
We consider first a subfamily of intervals where we have an estimate for the relative length of $(\tau_\lambda, \sigma_\lambda) \cap 
\mathcal A_\e$ in $(\tau_\lambda, \sigma_\lambda)$. 
Let $\Lambda_\e^{-}(N,\theta,\geq)\subseteq \Lambda_\e^{-}(N,\theta)$ be the subset of the indices $\lambda$ such that  
$$\sigma_\lambda-\tau_\lambda
\geq \frac{1}{2N^2S}\e;$$ 
that is, the condition in \eqref{mess}, which is always satisfied when $|z_\e^\lambda|\ge\theta$. 
For such indices we can reason as in the case $|z_\e^\lambda|\geq \theta$, again obtaining an estimate for the relative length of the intervals ``below the threshold''. More precisely,  
\begin{equation}\label{stima1s}
\sum_{\lambda\in\Lambda_\e^{-}(N,\theta, \geq)} F_\e(u_\e;(\tau_\lambda,\sigma_\lambda))\geq  C^{N}_\e \sum_{\lambda\in\Lambda_\e^{-}(N,\theta, \geq)} \mu_\e(\tau_\lambda, \sigma_\lambda) 
\end{equation}   
as in \eqref{stima1fe}. 

We now consider the set $\Lambda_\e^{-}(N,\theta,<)$ of the indices $\lambda\in\Lambda_\e^{-}(N,\theta)$ such that 
\begin{equation}\label{minore}\sigma_\lambda-\tau_\lambda
<\frac{1}{2N^2S}\e.
\end{equation}
In this case, we cannot prove that the measure of the intersection between $(\tau_\lambda, \sigma_\lambda)$  and $\mathcal A_\e$ is asymptotically negligible with respect to $|(\tau_\lambda, \sigma_\lambda)|$. 
Since \eqref{minore} holds, we can apply Lemma \ref{osck-lemma} to deduce that the number of indices 
$j\in\mathcal I_\e$ such that $I_j$ intersects $(\tau_\lambda,\sigma_\lambda)$ is equibounded by $k$.
For all indices $\lambda\in\Lambda_\e^{-}(N,\theta,<)$ 
we then only consider the intesecting intervals $I_j$ such that $j\in\mathcal I_\e^{r_k}\cap \mathcal I_\e^\ast(N)$, where $r_k$ is the threshold given by Lemma \ref{upperlemma} and $\mathcal I_\e^\ast(N)$ is defined by \eqref{definast}, obtaining that 
$$\#\{j\in\mathcal I_\e^{r_k}\cap \mathcal I_\e^\ast(N): I_j\subset (\tau_\lambda,\sigma_\lambda)\}\leq k.$$ 
Let $j_1(\lambda), \dots, j_{n_\lambda}(\lambda)$ denote these indices, 
and let $(a^{\ast,N}_{j_i(\lambda)},b^{\ast,N}_{j_i(\lambda)})\subset I_{j_i(\lambda)}$ be the corresponding maximal open subintervals such that $|u_\e^{(\ell)}|\leq\frac{1}{N\e^\ell}$  at the endpoints for all $\ell\in\{1,\dots, \ell(k)\}$. We can choose the indices in such a way that the finite sequences 
$\{a^{\ast,N}_{j_i(\lambda)}\}_i$ and $\{b^{\ast,N}_{j_i(\lambda)}\}_i$ are increasing. 
We can then write  
$$(\tau_\lambda,\sigma_\lambda)\setminus\bigcup_{i=1}^{n_\lambda}[a^{\ast,N}_{j_i(\lambda)},b^{\ast,N}_{j_i(\lambda)}]=\bigcup_{i=1}^{k_\lambda}
(\tau_\lambda^i,\sigma_\lambda^i),$$
where $k_\lambda\leq n_\lambda+1\leq k+1$ and the intervals $(\tau_\lambda^i,\sigma_\lambda^i)$ are pairwise disjoint. 
Note that each interval $(\tau_\lambda^i,\sigma_\lambda^i)$ is given by the union of 
\begin{enumerate}
\item[(i)] a subset $A_\lambda^i$ where $\int_{A_\lambda^i}|u_\e'|^2\, dt\geq \frac{r_k}{\e}|A_\lambda^i|$; 
\item[(ii)] a family of intervals $I_j$ with $j\in\mathcal I_\e^{r_k}\setminus \mathcal I_\e^\ast(N)$; 
\item[(iii)] (at most two) subsets of $I_j\setminus (a_j^{\ast,N},b_j^{\ast, N})$ with $j\in\mathcal I_\e^{r_k}\setminus \mathcal I_\e^\ast(N)$. 
\end{enumerate}
Lemma \ref{lemmar} and Remark \ref{bordir} ensure the existence of a threshold $\e_k(N)$ such that for all $\e\in(0,\e_k(N))$ 
\begin{equation*}
\sum_{\lambda\in\Lambda_\e^{-}(N,\theta,<)}\Big|(\tau_\lambda,\sigma_\lambda)\cap \bigcup_{j\in \mathcal I_\e^{r_k}\setminus \mathcal I_\e^\ast(N)}I_j\Big|\leq \sum_{j\in \mathcal I_\e^{r_k}\setminus \mathcal I_\e^\ast(N)}|I_j|\leq 
R_k k^2 S N^2\e^{\frac{k}{k-1}} \end{equation*}
and 
\begin{eqnarray*}
&&\sum_{\lambda\in\Lambda_\e^{-}(N,\theta,<)}\Big|(\tau_\lambda,\sigma_\lambda)\cap \bigcup_{j\in \mathcal I_\e^{r_k}\cap \mathcal I_\e^\ast(N)}\big(I_j\setminus (a_j^{\ast,N},b_j^{\ast,N})\big) \Big|\\
&&\hspace{2cm}\leq \sum_{j\in \mathcal I_\e^\ast(N)}|I_j\setminus (a_j^{\ast,N},b_j^{\ast,N})|\leq R_k k^2 S N^2\e^{\frac{k}{k-1}}.
\end{eqnarray*} 
By these estimates, since by construction $(\tau^i_{\lambda},\sigma^i_{\lambda})\cap\bigcup_{j\in\mathcal I_\e^{r_k}\cap \mathcal I_\e^\ast(N)} (a_j^{\ast,N},b_j^{\ast,N})=\emptyset$, we obtain that 
\begin{eqnarray*}
\sum_{\lambda\in\Lambda_\e^{-}(N,\theta,<)}\sum_{i=1}^{k_\lambda} \int_{\tau^i_{\lambda}}^{\sigma^i_{\lambda}}f_\e(u_\e^\prime)\, ds
&\geq&r_k
\sum_{\lambda\in\Lambda_\e^{-}(N,\theta,<)}\sum_{i=1}^{k_\lambda}\frac{\sigma^i_{\lambda}-\tau^i_{\lambda}}{\e}-R_\e^N, 
\end{eqnarray*} 
where $R_\e^N:=2R_k k^2 S N^2\e^{\frac{1}{k-1}}\to 0$ as $\e\to 0$. 
We then have 
\begin{equation}\label{stima2s}
\sum_{\lambda\in\Lambda_\e^{-}(N,\theta,<)}
\sum_{i=1}^{k_\lambda} F_\e(u_\e;(\tau^i_\lambda,\sigma^i_\lambda))\geq r_k\sum_{\lambda\in\Lambda_\e^{-}(N,\theta,<)}
\sum_{i=1}^{k_\lambda}\mu_\e(\tau_\lambda^i, \sigma_\lambda^i)-R_\e^{N}. 
\end{equation}

\smallskip 
Using these observations, we can now state an analog of Proposition \ref{propsopra}. 
To this end, we introduce the function $\psi_\e^{\theta,N}$ defined by
\begin{equation}\label{defpsitheta} 
\psi_\e^{\theta,N}(z)=
\begin{cases} r_k m_k^{\ell(k)}(N\theta)\theta^{\frac{1}{k}-\frac{3}{4}}\min\{\e^{-\frac{1}{4}}|z|, |z|^{\frac{3}{4}}\}
& \hbox{\rm if } |z|\leq\theta
\\
r_k m_k^{\ell(k)}(N\theta)|z|^{\frac{1}{k}}& \hbox{\rm if } |z|\geq\theta,
\end{cases} 
\end{equation} 
where 
$m_k^{\ell(k)}(\cdot)$ is the minimum defined in \eqref{defCN} with $\ell(k)$ given by \eqref{defellk}, and $r_k$ is the positive constant given by Lemma \ref{upperlemma}.

\begin{proposition}[Lower estimates in the case $|z_\e^\lambda|<\theta$]\label{propsotto} 
Let $\theta\in(0,1)$ be fixed. 
There exist an integer $N(\theta)$ and $\e_k(\theta,N)>0$ such that for all 
$N\geq N(\theta)$ and $\e\in(0,\e(\theta,N))$
the following estimates hold:  
\begin{equation}\label{step1}  
\sum_{\lambda\in\Lambda_\e^-(N, \theta,\geq)}F_\e(u_\e;(\tau_\lambda,\sigma_\lambda))\geq C^{N}_\e\sum_{\lambda\in\Lambda_\e^-(N,\theta,\geq)}\psi_\e^{\theta,N}(|u_\e(\sigma_\lambda)-u_\e(\tau_\lambda)|), 
\end{equation}
where $C^{N}_\e=1-
4kR_k(N)N^2S\e^{\frac{1}{2(k-2)}}$, and 
\begin{equation}\label{step1bis}  
\hskip-1mm\sum_{\lambda\in\Lambda_\e^-(N, \theta,<)}\!\sum_{i=1}^{k_\lambda}F_\e(u_\e;(\tau^i_\lambda,\sigma^i_\lambda))\geq\!\sum_{\lambda\in\Lambda_\e^-(N,\theta,<)}\!\sum_{i=1}^{k_\lambda}\psi_\e^{\theta,N}(|u_\e(\sigma^i_\lambda)-u_\e(\tau^i_\lambda)|)
-R^{N}_\e, 
\end{equation} 
where   $R^{N}_\e=2R_k k^2 S N^2\e^{\frac{1}{k-1}}$.
\end{proposition} 
\begin{proof}
Since both \eqref{stima1s} and \eqref{stima2s} hold for $\e<\min\{\e_k(N),\frac{1}{N^2}\}$, 
to prove the lower bound for the energy we now estimate $\mu_\e(\tau_\lambda,\sigma_\lambda)$ and $\mu_\e(\tau^i_\lambda,\sigma^i_\lambda)$ 
therein by a change of variable as in \eqref{scaling}. 
We consider a new scaling $\beta_\e$ since the scaling 
$\beta_\e=\e|z_\e^\lambda|^{\frac{1}{k}}$ does not imply that the derivatives at the endpoints are less than $\frac{1}{\theta N}$ as in \eqref{condiz} for $|z_\e^\lambda|$ ``small''. 

In the following computations, we will use the notation $(\tau,\sigma)$ to refer both to an interval $(\tau_\lambda, \sigma_\lambda)$ with $\lambda\in \Lambda^-_\e(N,\theta,\geq)$ and to an interval $(\tau^i_\lambda, \sigma^i_\lambda)$ with $\lambda\in \Lambda^-_\e(N,\theta,<)$ and $i\in\{1,\dots, k_\lambda\}$. The corresponding values of $u_\e(\sigma)-u_\e(\tau)$ will be called $z_\e$. 

Preliminarily, note that 
if $z_\e= 0$ then $\mu_\e(\tau,\sigma)=0$. 
We further subdivide the analysis in the two cases $|z_\e|\in[\e,\theta)$ and 
$|z_\e|<\e$. 

\smallskip 
\noindent{\em Case $\e\leq |z_\e|<\theta$.} 
We choose $\beta_\e=\theta^{\frac{1}{k}-\frac{3}{4}}\e|z_\e|^{\frac{3}{4}}$ as scaling factor in \eqref{scaling}. For any $v$ admissible test function for the minimum problem $\mu_\e(\tau,\sigma)$, letting $w$ be defined by  
$v(t)=z_\e w(\frac{t}{\beta_\e}-\frac{\tau+\sigma}{2\beta_\e})$ 
and $T_\e=\frac{\sigma-\tau}{\beta_\e}$, 
  we have 
\begin{eqnarray}\label{lessdelta}\nonumber
\frac{\sigma-\tau}{\e}+\e^{2k-1}\int_{\tau}^{\sigma} 
(v^{(k)})^2\, dt
&=&\theta^{\frac{1}{k}-\frac{3}{4}}|z_\e|^{\frac{3}{4}}\Big(T_\e
+\frac{|z_\e|^{2-\frac{3k}{2}}}{\theta^{2-\frac{3k}{2}}}\int_{-\frac{T_\e}{2}}^{\frac{T_\e}{2}} (w^{(k)})^2\, dt\Big)\nonumber\\
&\geq&\theta^{\frac{1}{k}-\frac{3}{4}}|z_\e|^{\frac{3}{4}}\Big(T_\e+\int_{-\frac{T_\e}{2}}^{\frac{T_\e}{2}} (w^{(k)})^2\, dt\Big)\nonumber\\
 &\geq& \theta^{\frac{1}{k}-\frac{3}{4}}|z_\e|^{\frac{3}{4}}m_k^{\ell(k)}(N) 
\end{eqnarray} 
for all $\e<\frac{1}{N^4}\theta^{3-\frac{4}{k}}$, since $w$ satisfies  
\begin{eqnarray*}
&&\Big|w^\prime\Big(\pm\frac{T_\e}{2}\Big)\Big|\leq\theta^{\frac{1}{k}-\frac{3}{4}}\frac{\sqrt\e}{|z_\e|^{\frac{1}{4}}}\leq \theta^{\frac{1}{k}-\frac{3}{4}}\e^{\frac{1}{4}}\\ 
&&\Big|w^{(\ell)}\Big(\pm\frac{T_\e}{2}\Big)\Big|\leq
\frac{\theta^{\frac{\ell}{k}}}{|z_\e|}\Big(\frac{|z_\e|}{\theta}\Big)^{\frac{3\ell}{4}}
\frac{1}{N}\leq \frac{1}{N}, \ \ \ell\in\{2,\dots, \ell(k)\}  
\end{eqnarray*} 
recalling that $|z_\e|<\theta\le 1$. 

\smallskip 

\noindent{\em Case $|z_\e|<\e$.} 
We choose $\beta_\e=\theta^{\frac{1}{k}-\frac{3}{4}}\e^{\frac{3}{4}} |z_\e|$. For any $v$ admissible test function for the minimum problem $\mu_\e(\tau,\sigma)$, with $w$ and $T_\e$ defined as above, we obtain 
\begin{eqnarray}\label{lesseps}
\frac{\sigma-\tau}{\e}+\e^{2k-1}\int_{\tau}^{\sigma} 
(v^{(k)})^2\, dt&=&\theta^{\frac{1}{k}-\frac{3}{4}}\frac{|z_\e|}{\e^{\frac{1}{4}}}\Big(T_\e+\frac{\theta^{\frac{3k}{2}-2}\e^{\frac{k}{2}}}{|z_\e|^{2k-2}}\int_{-\frac{T_\e}{2}}^{\frac{T_\e}{2}} (w^{(k)})^2\, dt\Big)\nonumber\\
&\geq&\theta^{\frac{1}{k}-\frac{3}{4}}\frac{|z_\e|}{\e^{\frac{1}{4}}}\Big(T_\e+\int_{-\frac{T_\e}{2}}^{\frac{T_\e}{2}} (w^{(k)})^2\, dt\Big)\nonumber\\
&\geq&\theta^{\frac{1}{k}-\frac{3}{4}}\frac{|z_\e|}{\e^{\frac{1}{4}}} m_k^{\ell(k)}(N), 
\end{eqnarray}
for all $\e<\frac{1}{N^4}\theta^{3-\frac{4}{k}}$, since $w$ satisfies 
\begin{eqnarray*}
&&
\Big|w^\prime\Big(\pm\frac{T_\e}{2}\Big)\Big|\leq\theta^{\frac{1}{k}-\frac{3}{4}}\e^{\frac{1}{4}}, \ \ 
\\ 
&&
\Big|w^{(\ell)}\Big(\pm\frac{T_\e}{2}\Big)\Big|\leq \frac{\theta^{\frac{\ell}{k}-\frac{3\ell}{4}}|z_\e|^{\ell-1}}{N\e^{
\frac{\ell}{4}}}
\leq \frac{1}{N}, \ \ \ell\in\{2,\dots, \ell(k)\} 
\end{eqnarray*} 
recalling that $|z_\e|<\e<\theta$.

Note that in both cases, if $(\tau,\lambda)=(\tau_\lambda, \sigma_\lambda)$ with $\lambda\in \Lambda_\e^-(N,\theta,\geq)$, the estimate on the derivatives holds up to the order $k-1$.

\medskip 

By \eqref{lessdelta} and \eqref{lesseps}, 
recalling the definition of the function $\psi_\e^{\theta,N}$ and the fact that $m_k^{\ell(k)}(\cdot)$ is increasing, we then deduce that 
\begin{equation}\label{stimamu1}
r_k\mu_\e(\tau_\lambda,\sigma_\lambda)\geq \psi_\e^{\theta,N}(|z^\lambda_\e|)  
\end{equation}
for all $\lambda\in\Lambda_\e^-(N,\theta,\geq)$, and that 
\begin{equation}\label{stimamu2}
r_k\mu_\e(\tau^i_\lambda,\sigma^i_\lambda)\geq \psi_\e^{\theta,N}(|u_\e(\sigma^i_\lambda)-u_\e(\tau^i_\lambda)|)  
\end{equation}
for all $\lambda\in\Lambda_\e^-(N,\theta,<)$ and $i\in\{1,\dots, k_\lambda\}$. Note that in this last inequality we have used the fact that $|u_\e(\sigma^i_\lambda)-u_\e(\tau^i_\lambda)|\le\theta$ thanks to Remark \ref{lenggen} and condition \eqref{minore}. 

By using \eqref{stimamu1} and \eqref{stimamu2} in \eqref{stima1s} and \eqref{stima2s}, respectively, and summing up, we obtain the claim with $\e_k(\theta, N)=\min\{\e_k(N), \frac{1}{N^4}\theta^{3-\frac{4}{k}}\}$. 
\end{proof}

\subsection{A compactness result}
We may now prove the following compactness theorem.  
\begin{theorem}[Compactness]\label{compactness} 
Let $\{u_\e\}$ be such that $F_\e(u_\e)\leq S<+\infty$. Then, there exists $u\in SBV_2(0,1)$  such that, up to subsequences and up to addition of suitable constants, $u_\e$ converge in measure to $u\in L^1(0,1)$. 
\end{theorem}

\begin{proof}
We fix $\theta\in(0,1)$ and $N\geq N(\theta)$.  
Let $\eta\in (0,1)$ be fixed. There exists $\e_k(\eta)>0$ such that, if $\e<\e_k(\eta)$, then 
$C_\e^{N}\geq (1-\eta)$ and $R_\e^N\leq\eta$ for all $N$.  
Since $m_k(N\theta)\geq r_k m_k^{\ell(k)}(N\theta)$, 
by \eqref{stima1fefin} and \eqref{step1} we deduce that  for all $N\geq N(\theta)$ and $\e<\min\{\e_k(\eta),\e_k(\theta, N)\}$
\begin{equation}\label{stimacompatt1}
\sum_{\substack{\lambda\in\Lambda_\e^+(N,\theta)\\
\lambda\in\Lambda_\e^-(N,\theta,\geq)}} 
F_\e(u_\e;(\tau_\lambda,\sigma_\lambda))\geq
(1-\eta)\sum_{\substack{\lambda\in\Lambda_\e^+(N,\theta)\\
\lambda\in\Lambda_\e^-(N,\theta,\geq)}}  \psi_\e^{\theta,N}(|u_\e(\sigma_\lambda)-u_\e(\tau_\lambda)|)
\end{equation}
and, using \eqref{step1bis},
\begin{eqnarray}\label{stimacompatt}
&&\hskip-2cm\sum_{\lambda\in\Lambda_\e^-(N,\theta,<)} 
\sum_{i=1}^{k_\lambda}F_\e(u_\e;(\tau^i_\lambda,\sigma^i_\lambda))\nonumber\\
&&\geq
\sum_{\lambda\in\Lambda_\e^-(N,\theta,<)}\sum_{i=1}^{k_\lambda} \psi_\e^{\theta,N}(|u_\e(\sigma^i_\lambda)-u_\e(\tau^i_\lambda)|)-\eta.
\end{eqnarray}

We now construct a sequence of functions  
$v_\e^{\theta, N}\in SBV_2(0,1)$ 
such that 
\begin{equation}\label{saltiv}
S(v_\e^{\theta, N})\subseteq\{\sigma_\lambda: \lambda\in \Lambda_\e(N)\}\cup \{\sigma_\lambda^i: \lambda\in \Lambda_\e^-(N,\theta,<), i\in\{1,\ldots,k_\lambda\}\},   
\end{equation} 
and such that 
\begin{enumerate}
\item[(i)] for all $\lambda\in \Lambda_\e^+(N,\theta)\cup \Lambda_\e^-(N,\theta, \geq)$  
\begin{equation}\label{saltiv2}
v_\e^{\theta,N}(\sigma_\lambda+)-v_\e^{\theta,N}(\sigma_\lambda-)=
u_\e(\sigma_\lambda)-u_\e(\tau_\lambda); 
\end{equation} 
\item[(ii)] for all $\lambda\in\Lambda_\e^-(N,\theta, <)$  and  $i\in\{1,\ldots,k_\lambda\}$
\begin{equation}\label{saltiv3}
v_\e^{\theta,N}(\sigma_\lambda^i+)-v_\e^{\theta,N}(\sigma_\lambda^i-)=
u_\e(\sigma_\lambda^i)-u_\e(\tau_\lambda^i). 
\end{equation} 
\end{enumerate}
We start by setting 
$$v_\e^{\theta, N}=u_\e \ \ \hbox{\rm in }\ \ 
\mathcal A_\e(N)\cup J(\e),$$ 
where we recall that $\mathcal A_\e(N)$ is defined in \eqref{defaen} and
$J(\e)=\{t\in(a_\e,b_\e)\setminus 
 \mathcal A_\e(N): \tau(t)=\sigma(t)\}$.
If $\lambda\in\Lambda_\e^{+}(N,\theta)\cup \Lambda_\e^{-}(N,\theta, \geq)$, we define 
\begin{equation*}\label{defve2}
v_\e^{\theta, N}(t)=u_\e(\tau_\lambda)\ \hbox{\rm in }\ \  (\tau_\lambda,\sigma_\lambda). 
\end{equation*} 
Then, we consider the intervals $(\tau_\lambda,\sigma_\lambda)$ with $\lambda\in\Lambda_\e^{-}(N,\theta,<)$. In this case, 
we set 
$$v_\e^{\theta, N}(t)=\begin{cases}
u_\e(\tau_\lambda^{i})&\hbox{\rm in } (\tau_\lambda^{i},\sigma_\lambda^{i}), \ i\in\{1,\dots, k_\lambda\} \\
u_\e(t)&\hbox{\rm otherwise in } (\tau_\lambda,\sigma_\lambda). 
\end{cases}$$ 
We finally extend $v_\e^{\theta, N}$ to the whole interval $(0,1)$ by continuity, 
setting 
\begin{equation*}\label{defve3}
v_\e^{\theta, N}(t)=
\begin{cases}
v_\e^{\theta, N}(a_\e+)&\hbox{\rm in } (0,a_\e)\\ 
v_\e^{\theta, N}(b_\e-)&\hbox{\rm in } (b_\e,1).
\end{cases}
\end{equation*}  
To prove that the set of the jump points of $v_\e^{\theta, N}$  satisfies  \eqref{saltiv}
we only have to check the continuity of $v_\e^{\theta, N}$ at any $\tau_\lambda$ and $\tau_\lambda^i$. 
By the definitions of $\tau(t)$, $\sigma(t)$ and $v_\e^{\theta, N}$, 
we have that 
$$\lim_{t\to\tau_\lambda^-}v_\e^{\theta, N}(t)=u_\e(\tau_\lambda)=v_\e^{\theta, N}(\tau_\lambda) \ \ \hbox{\rm and }\ \ \lim_{t\to\tau_\lambda^+}v_\e^{\theta, N}(t)=u_\e(\tau_\lambda)=v_\e^{\theta, N}(\tau_\lambda),$$
where in the last limit we used the fact that $v_\e^{\theta, N}$ is constantly equal to $u_\e(\tau_\lambda)$ in a right neighbourhood of  $\tau_\lambda$.  Similarly for $\tau_\lambda^i$. 
This allows to conclude that \eqref{saltiv} holds. 
By construction, \eqref{saltiv2} and \eqref{saltiv3} are also satisfied. 

We now define the set $J_\e^{N}$ as the union of the intervals $(\tau_\lambda,\sigma_\lambda)$ for $\lambda\in \Lambda_\e^+(N,\theta)\cup \Lambda_\e^-(N,\theta,\geq)$, of the intervals $(\tau^i_\lambda,\sigma^i_\lambda)$ for $\lambda\in\Lambda_\e^-(N,\theta,<)$ and $i\in\{1,\dots, k_\lambda\}$, and of the set $(0,1)\setminus (a_\e,b_\e)$. 
By using \eqref{saltiv2} and \eqref{saltiv3} in \eqref{stimacompatt1} and \eqref{stimacompatt}, we can conclude that 
\begin{eqnarray}\label{stimacomp}
F_\e(u_\e)
&\geq&(1-\eta)\sum_{t\in S(v^{\theta, N}_\e)}\!\!\psi_\e^{\theta,N}(|v_\e^{\theta,N}(t+)-v_\e^{\theta,N}(t-)|)\nonumber\\
&&+\int_{(0,1)\setminus J_\e^{N}} \!\!((v_\e^{\theta,N})^\prime)^2\, dt-\eta\nonumber\\
&\geq&(1-\eta)\sum_{t\in S(v_\e^{\theta,N})}\psi_\e^{\theta,N}(|v_\e^{\theta,N}(t+)-v^{\theta,N}_\e(t-)|)\nonumber\\
&&+\int_{0}^1 \!((v_\e^{\theta,N})^\prime)^2\, dt-\eta, 
\end{eqnarray}  
since $(v_\e^{\theta,N})^\prime=0$ in $J_\e^{N}$.

In order to apply Lemma \ref{bbblemma} to deduce the convergence, up to subsequences, of the sequence $\{v_\e^{\theta, N}\}$, we define the functional $G_\e^{\theta, N}$ by setting 
$$G_\e^{\theta, N}(v)=\sum_{t\in S(v)}
\psi_\e^{\theta,N}(|v(t+)-v(t-)|)+\int_{0}^1 (v^\prime)^2\, dt,$$ 
if $v\in SBV_2(0,1)$, and by $G_\e^{\theta, N}(v)=+\infty$ otherwise in $BV(0,1)$. 
Since the function $\psi_\e^{\theta,N}(z)$ is strictly increasing and concave in $[0,+\infty)$, and satisfies the conditions $\psi_\e^{\theta,N}(0)=0$ and $\lim\limits_{z\to+\infty}\psi_\e^{\theta,N}(z)=+\infty$, we can apply 
Lemma \ref{bbblemma} to the functional $G_\e^{\theta, N}$, so that, since by \eqref{stimacomp} 
\begin{equation}\label{stimaprecomp}
F_\e(u_\e)\geq(1-\eta)G_\e^{\theta, N}(v_\e^{\theta,N}) -\eta,
\end{equation}  
we deduce that the lower semicontinuous envelope of $G_\e^{\theta, N}$ with respect to the $L^1$-convergence is equibounded by $2S$. 
Hence, we deduce that 
there exists $u\in BV(0,1)$ such that  up to translations $v_\e^{\theta, N}\to u$ in $L^1(0,1)$ up to subsequences. 
Moreover, since $\lim_{\e\to 0}(\psi_\e^{\theta,N})^\prime(0)=+\infty$, 
we deduce that $u\in SBV_2(0,1)$. 

\smallskip 

We finally prove that, up to subsequences and up to translations, $u_\e\to u$ in measure as $\e\to 0$. 
Indeed, 
$$
    \{u_\e\neq v_\e^{\theta,N})\}\cap(a_\e,b_\e)\subseteq\bigcup_{\substack{\lambda\in \Lambda_\e^+(\theta, N)\\ \lambda\in \Lambda_\e^-(\theta, N,\geq)}}(\tau_\lambda,\sigma_\lambda)\cup\bigcup_{\lambda\in \Lambda_\e^-(\theta, N,<)}\bigcup_{i=1}^{k_\lambda}(\tau^i_\lambda, \sigma^i_\lambda).
$$ 
We recall that, by Proposition \ref{stimatausigma1}, 
if $\lambda\in \Lambda_\e^+(\theta, N)\cup \Lambda_\e^-(\theta, N,\geq)$ 
then we have
\begin{equation}\label{convue1}
\Big|(\tau_\lambda,\sigma_\lambda) \cap 
\mathcal A_\e\Big|
\leq 
4kR_k(N)N^2S\e^{\frac{1}{2(k-2)}}|(\tau_\lambda,\sigma_\lambda)|; 
\end{equation} 
since by \eqref{menoeps} 
\begin{equation*}
\sum_{\lambda\in\Lambda_\e(N)}\Big|(\tau_\lambda,\sigma_\lambda) \cap 
\Big\{t: |u_\e^\prime(t)|\geq \frac{1}{\sqrt\e}\Big\}\Big|
\leq S\e, 
\end{equation*} 
by \eqref{convue1} we deduce that 
\begin{equation}\label{convue1bis}
\sum_{\substack{\lambda\in \Lambda_\e^+(\theta, N)\\ \lambda\in \Lambda_\e^-(\theta, N,\geq)}}|(\tau_\lambda,\sigma_\lambda)|\leq \frac{S\e}{1-4kR_k(N)N^2S\e^{\frac{1}{2(k-2)}}}. 
\end{equation} 
We now consider the intervals $(\tau_\lambda,\sigma_\lambda)$ 
with $\lambda\in\Lambda_\e^-(N,\theta,<)$. 
By Lemma \ref{lemmar} and Remark \ref{bordir}, we obtain an estimate for the lengths of the subsets of the union of $(\tau_\lambda^i, \sigma_\lambda^i)$ where the $L^2$-norm of $u_\e^\prime$ is ``below the threshold'' $r_k$; that is, 
\begin{eqnarray*}
&&\sum_{\lambda\in\Lambda_\e^{-}(N,\theta,<)}\Big|(\tau_\lambda,\sigma_\lambda)\cap \bigcup_{j\in \mathcal I_\e^{r_k}\setminus \mathcal I_\e^\ast(N)}I_j\Big|\leq 
R_k k^2 S N^2\e^{\frac{k}{k-1}}\\
&&\sum_{\lambda\in\Lambda_\e^{-}(N,\theta,<)}\Big|(\tau_\lambda,\sigma_\lambda)\cap \bigcup_{j\in \mathcal I_\e^{r_k}\cap \mathcal I_\e^\ast(N)}\big(I_j\setminus (a_j^{\ast,N},b_j^{\ast,N})\big) \Big|\leq R_k k^2 S N^2\e^{\frac{k}{k-1}}. 
\end{eqnarray*} 
Noting that the remaining part of each $(\tau_\lambda^i, \sigma_\lambda^i)$ is given by sets $A_\lambda^i$ such that $\int_{A_\lambda^i}(u_\e^\prime)^2\, dt\geq\frac{r_k|A_\lambda^i|}{\e}$, 
recalling \eqref{stimalunghr} we deduce that the sum of the lengths of all these sets is estimated by $\frac{S\e}{r_k}$. 
Hence, we can conclude that 
\begin{equation}\label{convue2} 
\sum_{\lambda\in\Lambda_\e^{-}(N,\theta,<)}\sum_{i=1}^{k_\lambda}|(\tau_\lambda^i, \sigma_\lambda^i)|\leq R_k k^2 S N^2\e^{\frac{k}{k-1}}+\frac{S\e}{r_k}. 
\end{equation} 
Recalling that $(0,1)\setminus(a_\e,b_\e)=O(\e)$ as $\e\to 0$, we deduce that 
$$
|\{u_\e\neq v_\e^{\theta,N}\}|=o(1)$$ 
as $\e\to 0$. 
Hence, up to addition of constants and up to subsequences,  
the sequence 
$\{u_\e\}$ converges in measure to $u$, concluding the proof of the compactness.
\end{proof}

\subsection{Lower bound}

We now prove a sharp lower estimate for the $\Gamma$-limit by optimizing the use of the bounds given by Propositions \ref{propsopra} and \ref{propsotto}.  
\begin{theorem}[Lower bound] 
Let $u\in SBV_2(0,1)$. For all sequences $\{u_\e\}$ such that $u_\e\to u$ in measure, 
we have 
$$\liminf_{\e\to 0} F_\e(u_\e)\geq \int_0^1 (u^\prime)^2\, dt+m_k\sum_{t\in S(u)}|u(t+)-u(t-)|^{\frac{1}{k}}.$$
\end{theorem}

\begin{proof}
Up to subsequences, we can assume that $\sup_{\e>0}F_\e(u_\e)=S<+\infty$.   
For fixed $\theta\in (0,1)$ and $N\geq N(\theta)$, 
let $\{v^{\theta,N}_\e\}$ be the sequence constructed in the proof of Theorem \ref{compactness}, 
and note that $v^{\theta,N}_\e\to u$ in $L^1$.

In order to deduce an optimal bound by Propositions \ref{propsopra} and \ref{propsotto}, we set 
\begin{equation}
\phi_\e^{\theta,N}(z)=\begin{cases}
\psi_\e^{\theta,N}(z) & \hbox{\rm if } z<\theta\\
m_k(N\theta)|z|^{\frac{1}{k}} & \hbox{\rm if } z\geq\theta. 
\end{cases}
\end{equation}
Note that this function has a discontinuity in $\theta$.

If $\lambda\in\Lambda^-(N,\theta)$; that is, if $|u_\e(\tau_\lambda)-u_\e(\sigma_\lambda)|<\theta$, by \eqref{stimacompatt} we deduce the estimates 
\begin{equation*}
\sum_{\lambda\in\Lambda_\e^-(N,\theta,\geq)} 
F_\e(u_\e;(\tau_\lambda,\sigma_\lambda))\geq
(1-\eta)\sum_{\lambda\in\Lambda_\e^-(N,\theta,\geq)}  \psi_\e^{\theta,N}(|u_\e(\sigma_\lambda)-u_\e(\tau_\lambda)|)
\end{equation*}
and 
\begin{equation*}
\sum_{\lambda\in\Lambda_\e^-(N,\theta,<)}\!
\sum_{i=1}^{k_\lambda}F_\e(u_\e;(\tau^i_\lambda,\sigma^i_\lambda))\geq\!\sum_{\lambda\in\Lambda_\e^-(N,\theta,<)} \sum_{i=1}^{k_\lambda}\psi_\e^{\theta,N}(|
u_\e(\sigma^i_\lambda)-u_\e(\tau^i_\lambda))|)-\eta, 
\end{equation*}
with $N\geq N(\theta)$, $\eta>0$ arbitrary as in the proof of Theorem \ref{compactness},  and $\e<\min\{\e_k(\eta),\e_k(\theta, N)\}$, with $\e_k(\eta)$ and $\e_k(\theta, N)$ the positive thresholds given therein.  
By recalling \eqref{stima1fefin}, \eqref{step1} and \eqref{step1bis}, we then obtain 
\begin{equation}\label{stimapre}
F_\e(u_\e)
\geq
(1-\eta)\sum_{t\in S(v_\e^{\theta,N})} \phi_\e^{\theta,N}(|v_\e^{\theta,N}(t+)-v_\e^{\theta,N}(t-)|)+\int_0^1 ((v_\e^{\theta,N})^\prime)\, ds-\eta.  
\end{equation} 

We now define $L_\theta^{N}(z)$ as the linear function such that $L_\theta^{N}(0)=0$ and $L_\theta^{N}(\theta)=\psi_{\e}^{\theta, N}(\theta)
$; that is, 
$$
L_\theta^{N}(z)=r_km_k^{\ell(k)}(N\theta)(\theta)^{\frac{1}{k}-1}z.$$ 
If $\e<\theta$, then 
$$ \psi_{\e}^{\theta, N}(z)\geq L_{\theta}^N(|z|) \ \ \ \hbox{\rm if }\  |z|\leq \theta,$$ 
so that 
\begin{equation*}
\phi_{\e}^{\theta, N}(z)\geq \phi^{\theta, N}(z):=\min\{L_\theta^N(|z|), m_k(\lfloor N\theta\rfloor)|z|^{\frac{1}{k}}\}. 
\end{equation*}
Setting $\e_k(\eta, \theta, N)=\min\{\e_k(\eta),\e_k(\theta, N), \theta\}$, 
by \eqref{stimapre} we then obtain that,  
for all $\e\in (0,\e_k(\eta, \theta, N))$,  
\begin{equation*}
F_\e(u_\e)
\geq (1-\eta)G^{\theta,N}(v_\e^{\theta, N})-\eta, 
\end{equation*} 
where 
$$G^{\theta, N}(v)=
\sum_{t\in S(v)} \phi^{\theta,N}(|v(t+)-v(t-)|)+\int_0^1 (v^\prime)\, ds
$$
if $v\in SBV_2(0,1)$ and $+\infty$ otherwise. 
Applying again Lemma \ref{bbblemma} to the functional $G^{\theta, N}$, we obtain 
\begin{eqnarray*}
\liminf_{\e\to 0}F_\e(u_\e)&\geq&(1-\eta)\liminf_{\e\to 0}\overline G^{\theta, N}(v_\e^{\theta, N})-\eta\\
&\geq& (1-\eta)
\sum_{t\in S(u)} \phi^{\theta,N}(|u(t+)-u(t-)|)+ 
\int_0^1 \gamma^{\theta,N}(u^\prime)\, ds-\eta,  
\end{eqnarray*}
where 
$\gamma^{\theta,N}(z)=\min\{z^2, (\phi^{\theta,N})^\prime(0+)|z|\}^{\ast\ast}$. 
Since $\phi^{\theta,N}(z)$ is increasing with $N$, and $(\phi^{\theta,N})^\prime(0+)=r_km_k^{\ell(k)}(N\theta)\theta^{\frac{1}{k}-1}$ is also increasing, by taking the supremum over the admissible $N$ and using the monotone convergence, we deduce that 
\begin{eqnarray*}
\liminf_{\e\to 0}F_\e(u_\e)
\geq 
(1-\eta)\sum_{t\in S(u)} \phi^{\theta}(|u(t+)-u(t-)|)+ 
\int_0^1 \gamma^{\theta}(u^\prime)\, ds-\eta,  
\end{eqnarray*}
where $\phi^{\theta}(z):=\min\{r_km_k^{\ell(k)}\theta^{\frac{1}{k}-1}|z|, m_k|z|^{\frac{1}{k}}\}$ 
and $\gamma^{\theta}(z)=\min\{z^2, (\phi^{\theta})^\prime(0+)|z|\}^{\ast\ast}$. 
Again by monotonicity, by taking the supremum over $\theta$ we conclude that 
\begin{eqnarray*}
\liminf_{\e\to 0}F_\e(u_\e)
\geq (1-\eta)m_k
\sum_{t\in S(u)} |u(t+)-u(t-)|^{\frac{1}{k}}+ 
\int_0^1 (u^\prime)^2\, ds-\eta,  
\end{eqnarray*}
since $\sup_{\theta}\phi^{\theta}(z)=m_k|z|^{\frac{1}{k}}$ and $(\phi^{\theta})^\prime(0+)=r_km_k^{\ell(k)}\theta^{\frac{1}{k}-1}\to+\infty$ as $\theta\to 0$. 
By the arbitrariness of $\eta$, the claim follows. 
\end{proof}

\subsection{Upper bound}
We now conclude the proof of Theorem \ref{main} by showing the upper estimate. 
\begin{proposition}[Upper bound]\label{upb}
Let $u\in SBV_2(0,1)$. Then, there exists a sequence $\{u_\e\}$ such that $u_\e\to u$ in $L^1(0,1)$ and 
$$\limsup_{\e\to 0} F_\e(u_\e)\leq m_k\sum_{t\in S(u)}|u(t+)-u(t-)|^{\frac{1}{k}}+\int_{0}^1 (u^\prime)^2\, dt.$$
\end{proposition}

\begin{proof}
We start by proving the upper estimate for a function $u$ such that $S(u)=\{t_0\}$, $u\in C^\infty\big((0,1)\setminus\{t_0\}\big)$ and $u^{(\ell)}(t_0+)=u^{(\ell)}(t_0-)=0$ for any $\ell\geq 1$.  
Let $T>0$ and let $w\in H^k(-\frac{T}{2},\frac{T}{2})$ be such that $w$ is admissible for the minimum problem defining $m_k$, and 
$$m_k=T+\int_{-\frac{T}{2}}^{\frac{T}{2}}(w^{(k)})^2\, dt.$$
We set $z=u(t+)-u(t-)$ and define for all $\e>0$ 
\begin{equation}\label{def-ve}
v_\e(t)=z\, w\Big(\frac{t-t_0}{\e |z|^{\frac{1}{k}}}\Big).
\end{equation} 
Hence, 
\begin{eqnarray}\label{cvsup}
m_k&=&T+\e^{2k}\int_{-\frac{T}{2}}^{\frac{T}{2}}(v_\e^{(k)}(\e |z|^{\frac{1}{k}}t+t_0))^2\, dt\nonumber\\
&=&T+\frac{\e^{2k-1}}{|z|^{\frac{1}{k}}}\int_{t_0-\e|z|^{\frac{1}{k}}\frac{T}{2}}^{t_0+\e|z|^{\frac{1}{k}}\frac{T}{2}}(v_\e^{(k)})^2\, dt\nonumber\\
&=&\frac{1}{|z|^{\frac{1}{k}}}\bigg(T|z|^{\frac{1}{k}}+\e^{2k-1}\int_{t_0-\e|z|^{\frac{1}{k}}\frac{T}{2}}^{t_0+\e|z|^{\frac{1}{k}}\frac{T}{2}}(v_\e^{(k)})^2\, dt\bigg).
\end{eqnarray}
We can now define the recovery sequence $\{u_\e\}$ by setting 
\begin{equation}\label{def-ue}
u_\e(t)=\begin{cases}
u(t+\e|z|^{\frac{1}{k}} \frac{T}{2}) & \hbox{\rm if } t\in (0,t_0-\e|z|^{\frac{1}{k}} \frac{T}{2}]\\
v_\e(t) & \hbox{\rm if } t\in (t_0-\e|z|^{\frac{1}{k}} \frac{T}{2},t_0+\e|z|^{\frac{1}{k}} \frac{T}{2})\\
u(t-\e|z|^{\frac{1}{k}} \frac{T}{2})& \hbox{\rm if } t\in [t_0+\e \frac{T}{2},1), 
\end{cases}
\end{equation}
so that, thanks to \eqref{cvsup}, we get  
\begin{eqnarray*}
F_\e(u_\e)&\leq&\int_{t_0-\e|z|^{\frac{1}{k}}\frac{T}{2}}^{t_0+\e|z|^{\frac{1}{k}}\frac{T}{2}}\big(f_\e(v_\e)\, dt + \e^{2k-1}
(v_\e^{(k)})^2\big)\, dt\\
&&
+\int_0^1\big( (u^\prime)^2+\e^{2k-1} (u^{(k)})^2\big)\, dt\\
&\leq& 
T|z|^{\frac{1}{k}} + \e^{2k-1}
\int_{t_0-\e|z|^{\frac{1}{k}}\frac{T}{2}}^{t_0+\e|z|^{\frac{1}{k}}\frac{T}{2}}
(v_\e^{(k)})^2\, dt+\int_0^1 (u^\prime)^2\, dt+o(1)\\
&=&m_k|u(t_0+)-u(t_0-)|^{\frac{1}{k}}+\int_0^1 (u^\prime)^2\, dt+o(1),
\end{eqnarray*}
concluding the proof since $u_\e\to u$ in $L^1(0,1)$. 

The construction is completely analogous if $u$ is such that $\#S(u)<+\infty$, $u\in C^\infty\big((0,1)\setminus S(u)\big)$ and $u^{(\ell)}(t+)=u^{(\ell)}(t-)=0$ for any $\ell\geq 1$ and $t\in S(u)$. We then can conclude by a density argument since such a class is dense in $SBV_2(0,1)$.
\end{proof}

\section{Extensions}\label{sec:ext}
We conclude the paper by deriving a general convergence result as an envelope of functionals with truncated-quadratic potentials, and two applications showing new approximations for the Mumford--Shah and Blake--Zisserman functionals.

\subsection{An extension to a general class of integrands}\label{subsec:general}
We first state an extension, whose proof is exactly the same as in the case $a=b=c=1$ and is omitted. For notational simplicity we limit our analysis to functionals on $(0,1)$.
\begin{proposition} Let $a,b,c>0$ and define
\begin{equation*} 
F^{a,b,c}_\e(u)=\int_{(0,1)}\big(f^{a,b}_\e(u^\prime)+ c\, \e^{2k-1}(u^{(k)})^2\big)\, dt 
\end{equation*}
for $u\in H^k(0,1)$, with 
$f^{a,b}_\e(z)=\min\big\{ a z^2,\frac{b}{\e}\big\}$. Then the $\Gamma$-limit with respect to the convergence in measure and
the $L^1$-convergence of $F^{a,b,c}_\e$  is
\begin{equation*} 
F^{a,b,c}(u)= a\int_{(0,1)}(u')^2\,dt + m^{b,c}_k\sum_{t\in S(u)}|u(t+)-u(t-)|^{\frac{1}{k}},
\end{equation*}
with
\begin{eqnarray}\label{intro-10}\nonumber
&&m^{b,c}_k=\inf_{T>0}\min\bigg\{ bT + c\int_{-\frac{T}{2}}^{\frac{T}{2}} (v^{(k)})^2\,dt: v\in H^k\big(-\tfrac{T}{2},\tfrac{T}{2}\big),\ 
v\big(\pm\tfrac{T}{2}\big)= \pm \tfrac{1}{2},\\
&& \hskip3cm v^{(\ell)}\big(\pm\tfrac{T}{2}\big)=0 \hbox{ for all } \ell\in\{1,\ldots, k-1\}\Big\}.
\end{eqnarray} 
\end{proposition}

\begin{proposition} \label{main-gen}
Let $\alpha,\beta,\gamma>0$, and let $f$ be an increasing lower-semicontinuous function on $[0,+\infty)$ such that $$f(0)=0, \qquad f'(0)=\alpha\quad\hbox{ and }\quad\lim\limits_{z\to+\infty} f(z)=\beta,$$ then the $\Gamma$-limit with respect to the convergence in measure and
the $L^1$-convergence of 
\begin{equation*} 
F_\e(u)=\int_{(0,1)}\Big(\frac{1}{\e}f(\e|u^\prime|^2)+ \gamma \e^{2k-1}(u^{(k)})^2\Big)\, dt 
\end{equation*}
defined on $H^k(0,1)$ is $F^{\alpha,\beta,\gamma}$. 
\end{proposition}

\begin{proof} Note preliminarily that if $f(z)=\min\{a|z|, b\}$ then $F_\e=F^{a,b,\gamma}_\e$, so that the $\Gamma$-limit is given by
$F^{a,b,\gamma}$. Hence,  for a general $f$, if  $a,b>0$ are such that $\min\{a|z|, b\}\le f(z)$ for all $z$, then
$\Gamma\hbox{-}\liminf\limits_{\e\to 0} F_\e\ge F^{a,b,\gamma}$. We then have
$$
\Gamma\hbox{-}\liminf_{\e\to 0} F_\e(u)\ge \sup\Big\{a\int_{(0,1)}(u')^2\,dt 
+ m^{b,\gamma}_k\sum_{t\in S(u)}|u(t+)-u(t-)|^{\frac{1}{k}}\Big\},
$$
where the supremum is taken for all pairs $(a,b)$ such that $\min\{a|z|, b\}\le f(z)$ for all $z$.
Now, let 
\begin{eqnarray*}
&&\mathcal A=\{a>0: b>0 \hbox{ exists}: \min\{a|z|, b\}\le f(z) \hbox{ for all }z\}
\\
&&\mathcal B=\{b>0: a>0 \hbox{ exists}: \min\{a|z|, b\}\le f(z) \hbox{ for all }z\}.
\end{eqnarray*}
By the Supremum of Measures Lemma (see e.g.~\cite[Section 1.1.1]{bln}) we have 
\begin{eqnarray*}
&&\hskip-1.5cm\sup\Big\{a\int_{(0,1)}(u')^2\,dt + m^{b,\gamma}_k\sum_{t\in S(u)}|u(t+)-u(t-)|^{\frac{1}{k}}\Big\}
\\
&=&\Big(\sup_{a\in\mathcal A} a\Big)\int_{(0,1)}(u')^2\,dt + \Big(\sup_{b\in\mathcal B}m^{b,\gamma}_k\Big)
\sum_{t\in S(u)}|u(t+)-u(t-)|^{\frac{1}{k}}
\\
&=&\alpha\int_{(0,1)}(u')^2\,dt + m^{\beta,\gamma}_k\sum_{t\in S(u)}|u(t+)-u(t-)|^{\frac{1}{k}}\\
& =& F^{\alpha,\beta,\gamma}(u).
\end{eqnarray*}
To check the upper bound it suffices to consider the case of a function $u$ such that $S(u)=\{t_0\}$, $u\in C^\infty\big((0,1)\setminus\{t_0\}\big)$ and $u^{(\ell)}(t_0+)=u^{(\ell)}(t_0-)=0$ for any $\ell\geq 1$. We take $u_\e$ defined as in 
\eqref{def-ue}, where $v_\e$ is defined as in \eqref{def-ve} with $(T,w)$ a minimal pair for $m^{\beta,\gamma}_k$.
Since $f(z)\le \beta$, we then have
$$
F_\e(u_\e)\le \int_{(0,1)}\frac{1}{\e}f(\e|u^\prime|^2)\,dt+m^{\beta,\gamma}_k |u(t_0+)-u(t_0-)|^{\frac{1}{k}}\,.
$$ 
Since $\lim\limits_{\e\to 0} \frac{1}{\e}f(\e|u^\prime(t)|^2)=\alpha |u^\prime(t)|^2$ for all $t$, we can pass to the limit by using the Dominated Convergence Theorem, and obtain that 
$$
\Gamma\hbox{-}\limsup_{\e\to 0} F_\e(u)\le F^{\alpha,\beta,\gamma}(u).
$$
For a general $u\in SBV_2(0,1)$ we can use a density argument exactly as in the proof of Proposition \ref{upb}.
\end{proof}

\subsection{An approximation of the Mum\-ford--Shah functional}
Using a diagonal argument we may prove an approximation of the Mum\-ford--Shah functional as follows.

\begin{proposition} 
Let $\mu>0$. There exist $\e_k$ and $c_k$ such that, if $\Phi_k=  F^{1,1,c_k}_{\e_k}$, then 
then the $\Gamma$-limit of $\Phi_k$ with respect to the convergence in measure and
the $L^1$-convergence is given by 
\begin{equation}\label{intro-12}
\Gamma\hbox{-}\lim_{k\to+\infty}\Phi_k(u)=
\int_{(0,1)}(u')^2\,dt +\mu\#S(u) =: F^{MS}_\mu(u)
\end{equation}
on $SBV_2(0,1)$. 
\end{proposition}

\begin{proof}
Let $c_k$ be such that $m^{1,c_k}_k=\mu$. The existence of such $c_k$ follows 
from the fact that the function $c\mapsto m^{1,c}_k$ is continuous and by \eqref{intro-10} we have $\lim\limits_{c\to 0^+}m^{1,c}_k=0$ and $\lim\limits_{c\to +\infty}m^{1,c}_k=+\infty$. 
We then have 
\begin{eqnarray*}
\Gamma\hbox{-}\lim_{\e\to0}F^{1,1,c_k}_{\e}(u)&=&
\int_{(0,1)}(u')^2\,dt + \mu \sum_{t\in S(u)}|u(t+)-u(t-)|^{\frac{1}{k}}\\
&\geq&\int_{(0,1)}(u')^2\,dt + \mu \sum_{t\in S(u)}\min\{1, |u(t+)-u(t-)|^{\frac{1}{k}}\}. 
\end{eqnarray*}
Recalling that $|z|^{\frac{1}{k}}\to 1$ as $k\to+\infty$ for all $z\neq 0$, and this convergence is increasing for $|z|\leq 1$, 
we obtain as a consequence that 
$$\Gamma\hbox{-}\liminf_{k\to+\infty}\Big(\Gamma\hbox{-}\lim_{\e\to0}F^{1,1,c_k}_{\e}\Big)(u)\geq \int_{(0,1)}(u')^2\,dt +\mu\#S(u).$$ 
On the other hand, since the $\Gamma$-$\limsup$ is estimated from above by the pointwise limit, it satisfies the converse inequality, so that there exists 
$$\Gamma\hbox{-}\lim_{k\to+\infty}\Big(\Gamma\hbox{-}\lim_{\e\to0}F^{1,1,c_k}_{\e}\Big)(u)=\int_{(0,1)}(u')^2\,dt +\mu\#S(u).$$ 
Since the family of functionals $\{F^{1,1,c_k}_{\e}\}$ is equicoercive as $\e\to 0$, thanks to \cite[Theorem 17.14]{DM} we can use a diagonal argument and conclude. 
\end{proof}

\subsection{Approximation of the Blake-Zisserman functional}
We finally investigate the possibility of using the same type of approximation for the Blake--Zisserman functional \cite{BZ}, defined on piecewise-$H^2$ functions as 
\begin{equation}
F^{BZ} (u)=\int_0^1|u''|^2dt+ 2\#(S(u))+ \#(S(u')\setminus S(u))
\end{equation}
(again taking $(0,1)$ as the domain of the functions for simplicity). 
In this notation, $u'$ and $u''$ are the almost everywhere defined first and second weak derivative of $u$, and $S(u)$ and $S(u')$ are the discontinuity sets of $u$ and $u'$, respectively. The reason why the coefficient of the number of points of the  jump set $\#(S(u))$ is twice that of $\#(S(u')\setminus S(u))$ ({\em crease set}) is a lower-semicontinuity issue, since jump points can be approximated by crease points.
The functional $F^{BZ}$ can be seen to some extent as a higher-order version of the Mumford--Shah functional, and can also be defined in dimension $d>1$ on suitably defined spaces (see the review paper \cite{CLT} for details). We prove the following result.

\begin{theorem}\label{main-BZ}
Let $c_k$ be implicitly defined by 
\begin{eqnarray}\nonumber
&&\inf_{T>0}\min\bigg\{T+c_k\int_0^T|v^{(k)}|^2dt:
 v(0)=0, v(T)=1, \\ \label{mkck}
&&\hskip2cm v^{(\ell)}(0)=v^{(\ell)}(T)=0, \ \ell\in\{1,\ldots,k-1\}\Big\}=1,
\end{eqnarray}
and let $f_\e(z)=\min\{z^2,\frac1\e\}$.
 Then there exists a family $\overline\e_k$  tending to $0$ such that if $\e_k\le \overline\e_k$, then the functionals $G_k$ defined by 
\begin{equation}
G_k (u)=\int_0^1f_{\e_k}(u'')dt+c_{k-1}{\e_k}^{2k-3}\int_0^1|u^{(k)}|^2dt
\end{equation}
on $H^k(0,1)$  $\Gamma$-converge with respect to the $L^2$-convergence to $F^{BZ}$. 
\end{theorem}

\begin{proof}We note that, replacing $k$ with $k-1$, we can equivalently define the functionals $G_k$ as
\begin{equation}
G_k (u)=\int_0^1f_{\e_k}(u'')dt+c_k{\e_k}^{2k-1}\int_0^1|u^{(k+1)}|^2dt,
\end{equation}
so that they coincide with the functional $F_{k,\e_k}$,
computed on $u'$, where
\begin{equation}\label{Fkek}
F_{k,\e} (u)=\int_0^1f_{\e}(u')dt+c_k{\e}^{2k-1}\int_0^1|u^{(k)}|^2dt.
\end{equation}
We can suppose that $\e_k$ is such that $F_{k,\e_k}$ tend to $F^{MS}=F^{MS}_1$ as defined in \eqref{intro-12}.

\smallskip
{\em Upper bound.} Let $u$ be a piecewise-$H^2$ function with $S(u)=\emptyset$. 
Then the function $v=u'$ is piecewise $H^1$ and we can take a recovery sequence $v_k$ for $u$ along the sequence $F_k$, so that
$$
F^{BZ}(u)= F^{MS}(v)=\lim_{k\to+\infty} F_k(v_k)= \lim_{k\to+\infty} G_k(u_k),
$$
where $u_k$ is a primitive of $v_k$ converging to $u$. We then have 
$$
\Gamma\hbox{-}\limsup_{k\to+\infty} G_k(u)\le F^{BZ}(u)
$$
for all piecewise-$H^2$ functions $u$ with $S(u)=\emptyset$. 
Since the functional $F^{BZ} $ is the lower-semicontinuous envelope of its restriction 
\begin{equation}
\int_0^1|u''|^2dt+ \#(S(u'))
\end{equation}
to piecewise-$H^2$ functions $u$ with $S(u)=\emptyset$, we then deduce that the upper inequality holds for all piecewise-$H^2$ functions.

\smallskip
{\em Lower bound.} 
We note that, if we set 
\begin{equation}
G_{k,\e} (u)=\int_0^1f_{\e}(u'')dt+c_k{\e}^{2k-1}\int_0^1|u^{(k+1)}|^2dt,
\end{equation}
then for all $k$ the $\Gamma$-limit of such energies is finite only on $H^2$-functions $u$ with $u'$ piecewise $H^1$; that is, piecewise $H^2$-functions with $S(u)=\emptyset$, and we have
\begin{eqnarray*}
\Gamma\hbox{-}\lim_{\e\to 0} G_{k,\e} (u)&=& \int_0^1|u''|^2dt+\sum_{t\in S(u')}|u'(t+)-u'(t-)|^{1/k}
\\
&\ge &\int_0^1|u''|^2dt+\sum_{t\in S(u')}\min\{1,|u'(t+)-u'(t-)|^{1/k}\}.
\end{eqnarray*} 
Since  the lower-semicontinuous envelope of the functional defined on the family of piecewise-$H^2$ functions with $S(u)=\emptyset$ by the
$$
\int_0^1|u''|^2dt+\sum_{t\in S(u')}\min\{1,|u'(t+)-u'(t-)|^{1/k}\}
$$ is the functional defined on piecewise $H^2$-functions by
$$
\int_0^1|u''|^2dt+\sum_{t\in S(u')}\min\{1,|u'(t+)-u'(t-)|^{1/k}\}+ 2\#(S(u)),
$$ the lower-semicontinuous envelope of the functional defined on piecewise $H^2$-functions with $S(u)=\emptyset$ by the last expression is 
$$
\int_0^1|u''|^2dt+\sum_{t\in S(u')}\min\{1,|u'(t+)-u'(t-)|^{1/k}\}+ 2\#(S(u))
$$
(see e.g.~\cite{B-creases}). Taking the limit as $k\to+\infty$ we then have
$$
\Gamma\hbox{-}\lim_{k\to+\infty} \Big(\Gamma\hbox{-}\lim_{\e\to 0} G_{k,\e}\Big) (u)\ge F^{BZ}(u).
$$
and the claim, up to possibly taking smaller $\e_k$.
\end{proof}

\medskip 

\noindent{\bf Acknowledgements.} The author thanks the anonymous referees for their careful reading and insightful suggestions. 
The author acknowledges the INdAM group GNAMPA for funding her research in the framework of the project {\em ``Asymptotic analysis of nonlocal variational problems''} and the Department of Architecture, Design and Planning of the University of Sassari for the contribution in the framework of the projects {\em DM737/2021 and DM737/2022-23}. She is a member of the INdAM group GNAMPA.

\bibliographystyle{abbrv}

\bibliography{references}

\begin{thebibliography}{10}

\bibitem{ABS}
G.~Alberti, G.~Bouchitt\'e, and P.~Seppecher.
\newblock Phase transition with the line-tension effect.
\newblock {\em Arch. Rational Mech. Anal.}, 144(1):1--46, 1998.

\bibitem{abg}
R.~Alicandro, A.~Braides, and M.~S. Gelli.
\newblock Free-discontinuity problems generated by singular perturbation.
\newblock {\em Proc. Roy. Soc. Edinburgh Sect. A}, 128(6):1115--1129, 1998.

\bibitem{ABSS}
R.~Alicandro, A.~Braides, M.~Solci, and G.~Stefani.
\newblock Topological singularities arising from fractional-gradient energies,
  eprint ar{X}iv:2309.10112, 2024.

\bibitem{ag}
R.~Alicandro and M.~S. Gelli.
\newblock Free discontinuity problems generated by singular perturbation: the
  {$n$}-dimensional case.
\newblock {\em Proc. Roy. Soc. Edinburgh Sect. A}, 130(3):449--469, 2000.

\bibitem{AGS}
L.~Ambrosio, N.~Gigli, and G.~Savar\'e.
\newblock {\em Gradient flows in metric spaces and in the space of probability
  measures}.
\newblock Lectures in Mathematics ETH Z\"urich. Birkh\"auser Verlag, Basel,
  second edition, 2008.

\bibitem{MR2162866}
M.~Arndt and M.~Griebel.
\newblock Derivation of higher order gradient continuum models from atomistic
  models for crystalline solids.
\newblock {\em Multiscale Model. Simul.}, 4(2):531--562, 2005.

\bibitem{Bach}
A.~Bach.
\newblock Anisotropic free-discontinuity functionals as the {$\Gamma$}-limit of
  second-order elliptic functionals.
\newblock {\em ESAIM Control Optim. Calc. Var.}, 24(3):1107--1140, 2018.

\bibitem{MR1250554}
S.~Bardenhagen and N.~Triantafyllidis.
\newblock Derivation of higher order gradient continuum theories in {$2,3$}-{D}
  nonlinear elasticity from periodic lattice models.
\newblock {\em J. Mech. Phys. Solids}, 42(1):111--139, 1994.

\bibitem{bcg}
G.~Bellettini, A.~Chambolle, and M.~Goldman.
\newblock The {$\Gamma$}-limit for singularly perturbed functionals of
  {P}erona--{M}alik type in arbitrary dimension.
\newblock {\em Math. Models Methods Appl. Sci.}, 24(06):1091--1113, 2014.

\bibitem{bf}
G.~Bellettini and G.~Fusco.
\newblock The {$\Gamma$}-limit and the related gradient flow for singular
  perturbation functionals of {P}erona--{M}alik type.
\newblock {\em Trans. Amer. Math. Soc.}, 360(9):4929--4987, 2008.

\bibitem{BBH}
F.~Bethuel, H.~Brezis, and F.~H\'elein.
\newblock {\em Ginzburg-{L}andau {V}ortices}, volume~13 of {\em Progress in
  Nonlinear Differential Equations and their Applications}.
\newblock Birkh\"auser Boston, Inc., Boston, MA, 1994.

\bibitem{BZ}
A.~Blake and A.~Zisserman.
\newblock {\em Visual {R}econstruction}.
\newblock The MIT Press, Cambridge, 1987.

\bibitem{BLL}
X.~Blanc, C.~Le~Bris, and P.-L. Lions.
\newblock From molecular models to continuum mechanics.
\newblock {\em Arch. Ration. Mech. Anal.}, 164(4):341--381, 2002.

\bibitem{BBL}
M.~Bonnivard, E.~Bretin, and A.~Lemenant.
\newblock Numerical approximation of the {S}teiner problem in dimension 2 and
  3.
\newblock {\em Math. Comp.}, 89(321):1--43, 2020.

\bibitem{bbb}
G.~Bouchitt{\'e}, G.~Buttazzo, and A.~Braides.
\newblock Relaxation results for some free discontinuity problems.
\newblock {\em J. Reine Angew. Math.}, 458:1--18, 1995.

\bibitem{BDS}
G.~Bouchitt\'e, C.~Dubs, and P.~Seppecher.
\newblock Regular approximation of free-discontinuity problems.
\newblock {\em Math. Models Methods Appl. Sci.}, 10(7):1073--1097, 2000.

\bibitem{B-creases}
A.~Braides.
\newblock Lower semicontinuity conditions for functionals on jumps and creases.
\newblock {\em SIAM J. Math. Anal.}, 26(5):1184--1198, 1995.

\bibitem{bln}
A.~Braides.
\newblock {\em Approximation of Free-{D}iscontinuity Problems}.
\newblock Springer, Ber\-lin, 1998.

\bibitem{BCST}
A.~Braides, A.~Causin, M.~Solci, and L.~Truskinovsky.
\newblock Beyond the classical {C}auchy-{B}orn rule.
\newblock {\em Arch. Ration. Mech. Anal.}, 247(6):Paper No. 107, 113, 2023.

\bibitem{BDM}
A.~Braides and G.~Dal~Maso.
\newblock Non-local approximation of the {M}umford--{S}hah functional.
\newblock {\em Calc. Var. Partial Differential Equations}, 5(4):293--322, 1997.

\bibitem{BLO}
A.~Braides, A.~J. Lew, and M.~Ortiz.
\newblock Effective cohesive behavior of layers of interatomic planes.
\newblock {\em Arch. Ration. Mech. Anal.}, 180(2):151--182, 2006.

\bibitem{BMS}
A.~Braides, M.~Maslennikov, and L.~Sigalotti.
\newblock Homogenization by blow-up.
\newblock {\em Applicable Anal.}, 87:1341--1356, 2008.

\bibitem{BS-book}
A.~Braides and M.~Solci.
\newblock {\em Geometric flows on planar lattices}.
\newblock Pathways in Mathematics. Birkh\"auser/Springer, Cham, [2021]
  \copyright 2021.

\bibitem{BT}
A.~Braides and L.~Truskinovsky.
\newblock Asymptotic expansions by {$\Gamma$}-convergence.
\newblock {\em Contin. Mech. Thermodyn.}, 20(1):21--62, 2008.

\bibitem{BruDoS}
G.~C. Brusca, D.~Donati, and M.~Solci.
\newblock Higher-order singular perturbation models for phase transitions.
\newblock {\em SIAM J. Math. Anal.}, 57(3):3146--3170, 2025.

\bibitem{BEZ}
M.~Burger, T.~Esposito, and C.~I. Zeppieri.
\newblock Second-order edge-penalization in the {A}mbrosio-{T}ortorelli
  functional.
\newblock {\em Multiscale Model. Simul.}, 13(4):1354--1389, 2015.

\bibitem{CLT}
M.~Carriero, A.~Leaci, and F.~Tomarelli.
\newblock A survey on the {B}lake-{Z}isserman functional.
\newblock {\em Milan J. Math.}, 83(2):397--420, 2015.

\bibitem{cha}
A.~Chambolle.
\newblock Image segmentation by variational methods: {M}umford and {S}hah
  functional and the discrete approximations.
\newblock {\em SIAM J. Appl. Math.}, 55(3):827--863, 1995.

\bibitem{CT1}
M.~Charlotte and L.~Truskinovsky.
\newblock Lattice dynamics from a continuum viewpoint.
\newblock {\em J. Mech. Phys. Solids}, 60:1508--1544, 2012.

\bibitem{CDFL}
M.~Chermisi, G.~Dal~Maso, I.~Fonseca, and G.~Leoni.
\newblock Singular perturbation models in phase transitions for second-order
  materials.
\newblock {\em Indiana Univ. Math. J.}, 60(2):367--409, 2011.

\bibitem{CSZ}
M.~Cicalese, E.~N. Spadaro, and C.~I. Zeppieri.
\newblock Asymptotic analysis of a second-order singular perturbation model for
  phase transitions.
\newblock {\em Calc. Var. Partial Differential Equations}, 41(1-2):127--150,
  2011.

\bibitem{CFI1}
S.~Conti, M.~Focardi, and F.~Iurlano.
\newblock Phase field approximation of cohesive fracture models.
\newblock {\em Ann. Inst. H. Poincar\'e{} C Anal. Non Lin\'eaire},
  33(4):1033--1067, 2016.

\bibitem{cfi}
S.~Conti, M.~Focardi, and F.~Iurlano.
\newblock Approximation of {$SBV$} functions with possibly infinite jump set,
  eprint ar{X}iv 2309.16557, 2023.

\bibitem{CFI2}
S.~Conti, M.~Focardi, and F.~Iurlano.
\newblock Phase-field approximation of a vectorial, geometrically nonlinear
  cohesive fracture energy.
\newblock {\em Arch. Ration. Mech. Anal.}, 248(2):Paper No. 21, 60, 2024.

\bibitem{DM}
G.~Dal~Maso.
\newblock {\em An Introduction to $\Gamma$-convergence}.
\newblock Birkh\"auser, Boston, 1993.

\bibitem{FLL}
I.~Fonseca, P.~Liu, and X.~Lu.
\newblock Higher order {A}mbrosio--{T}ortorelli scheme with non-negative
  spatially dependent parameters.
\newblock {\em Adv. Calc. Var.}, 16(4):\ 885--902, 2023.

\bibitem{fm}
I.~Fonseca and C.~Mantegazza.
\newblock Second order singular perturbation models for phase transitions.
\newblock {\em SIAM J. Math. Anal.}, 31(5):1121--1143, 2000.

\bibitem{fmu}
I.~Fonseca and S.~M\"uller.
\newblock Quasi-convex integrands and lower semicontinuity in {$L^1$}.
\newblock {\em SIAM J. Math. Anal.}, 23(5):1081--1098, 1992.

\bibitem{go}
M.~Gobbino.
\newblock Finite difference approximation of the {M}umford--{S}hah functional.
\newblock {\em Comm. Pure Appl. Math.}, 51(2):197--228, 1998.

\bibitem{gomo}
M.~Gobbino and M.~G. Mora.
\newblock Finite-difference approximation of free-dis\-con\-ti\-nuity problems.
\newblock {\em Proc. Roy. Soc. Edinburgh Sect. A}, 131(3):567--595, 2001.

\bibitem{GP-2}
M.~Gobbino and N.~Picenni.
\newblock A quantitative variational analysis of the staircasing phenomenon for
  a second order regularization of the {P}erona-{M}alik functional.
\newblock {\em Trans. Amer. Math. Soc.}, 376(8):5307--5375, 2023.

\bibitem{Ki}
S.~Kichenassamy.
\newblock The {P}erona-{M}alik paradox.
\newblock {\em SIAM J. Appl. Math.}, 57(5):\ 1328--1342, 1997.

\bibitem{leoni}
G.~Leoni.
\newblock {\em A {F}irst {C}ourse in {S}obolev {S}paces}, volume 181 of {\em
  Graduate Studies in Mathematics}.
\newblock American Mathematical Society, Providence, RI, 2017.

\bibitem{Modica}
L.~Modica.
\newblock The gradient theory of phase transitions and the minimal interface
  criterion.
\newblock {\em Arch. Rational Mech. Anal.}, 98(2):123--142, 1987.

\bibitem{mo}
M.~Morini.
\newblock Sequences of singularly perturbed functionals generating
  free-discontinuity problems.
\newblock {\em SIAM J. Math. Anal.}, 35(3):759--805, 2003.

\bibitem{MS}
D.~Mumford and J.~Shah.
\newblock Optimal approximations by piecewise smooth functions and associated
  variational problems.
\newblock {\em Comm. Pure Appl. Math.}, 42(5):577--685, 1989.

\bibitem{PT}
L.~A. Peletier and W.~C. Troy.
\newblock {\em Spatial {P}atterns. Higher {O}rder {M}odels in {P}hysics and
  {M}echanics}, volume~45 of {\em Progress in Nonlinear Differential Equations
  and their Applications}.
\newblock Birkh\"auser Boston, Inc., Boston, MA, 2001.

\bibitem{CT2}
O.~U. Salman and L.~Truskinovsky.
\newblock De-localizing brittle fracture.
\newblock {\em J. Mech. Phys. Solids}, 154:104517, 2021.

\bibitem{SS}
E.~Sandier and S.~Serfaty.
\newblock {\em Vortices in the {M}agnetic {G}inzburg-{L}andau {M}odel},
  volume~70 of {\em Progress in Nonlinear Differential Equations and their
  Applications}.
\newblock Birkh\"auser Boston, Inc., Boston, MA, 2007.

\bibitem{SV}
O.~Savin and E.~Valdinoci.
\newblock {$\Gamma$}-convergence for nonlocal phase transitions.
\newblock {\em Ann. Inst. H. Poincar\'e{} C Anal. Non Lin\'eaire},
  29(4):479--500, 2012.

\end{thebibliography}

\end{document}